\documentclass[11pt]{amsart}

\usepackage{latexsym}
\usepackage{amssymb}
\usepackage{amsthm}
\usepackage{amsmath}
\usepackage{amsfonts}
\usepackage{mathrsfs}

\usepackage[all]{xy}

\newtheorem{thm}{Theorem}[section]        \newtheorem{lem}[thm]{Lemma}
\newtheorem{cor}[thm]{Corollary}      \newtheorem{prop}[thm]{Proposition}
        
\newtheorem*{defi*}{Definition}  \newtheorem*{ex*}{Example}
\newtheorem*{thm*}{Theorem}      \newtheorem*{cor*}{Corollary}
\newtheorem*{rmk*}{Remark}      
    \newtheorem*{prop*}{Proposition}
\newtheorem{claim}[thm]{Claim}

\newcommand{\ds}{\displaystyle}

\let\ssection=\section
\renewcommand{\section}{\setcounter{equation}{0}\ssection}

\newtheorem*{namedtheorem}{\theoremname}
\newcommand{\theoremname}{testing}

\theoremstyle{remark}
\newtheorem*{rem}{Remark}

\newcommand{\BR}{\mathbb R}      
      
      \newcommand{\BZ}{\mathbb Z}

\newcommand{\CC}{\mathcal C}    \newcommand{\CD}{\mathcal D}
    \newcommand{\CF}{\mathcal F}

\newcommand{\D}{\partial}

\newcommand{\ep}{\epsilon}
\newcommand{\cover}{\widetilde}
\newcommand{\closure}{\overline}

\renewcommand{\phi}{\varphi}
\renewcommand{\bar}{\overline}

\newcommand{\FN}{F_N}
\newcommand{\ssm}{\smallsetminus} 
\renewcommand{\b}{\textbf}

\DeclareMathOperator{\Out}{Out}    

\begin{document}

\title[]{Ergodic decomposition for folding and unfolding paths in Outer space}
\author{Hossein Namazi, Alexandra Pettet, and Patrick Reynolds}
\begin{abstract}
We relate ergodic-theoretic properties of a very small tree or lamination to the behavior of folding and unfolding paths in Outer space that approximate it, and we obtain a criterion for unique ergodicity in both cases.  Our main result is that non-unique ergodicity gives rise to a transverse decomposition of the folding/unfolding path.  It follows that non-unique ergodicity leads to distortion when projecting to the complex of free factors, and we give two applications of this fact.  First, we show that if a subgroup $H$ of $Out(\FN)$ quasi-isometrically embeds into the complex of free factors via the orbit map, then the limit set of $H$ in the boundary of Outer space consists of trees that are uniquely ergodic and have uniquely ergodic dual lamination.  Second, we describe the Poisson boundary for random walks coming from distributions with finite first moment with respect to the word metric on $Out(\FN)$: almost every sample path converges to a tree that is uniquely ergodic and that has a uniquely ergodic dual lamination, and the corresponding hitting measure on the boundary of Outer space is the Poisson boundary.  This improves a recent result of Horbez.  We also obtain sublinear tracking of sample paths with Lipschitz geodesic rays.  
\end{abstract}
\maketitle

\section{Introduction}\label{sec: introduction}

In \cite{Ma92} Masur proved that the vertical foliation for a quadratic differential is uniquely ergodic if a corresponding Teichm\"uller geodesic ray projects to a recurrent ray in the moduli space. This observation which is known as Masur Criterion provides the first step in understanding the dynamical behavior of Teichm\"uller geodesics and has had numerous applications. In particular Kaimanovich-Masur \cite{KM96} used this to study random walks in the mapping class group of a surface and in Teichm\"uller space. They showed that for a non-elementary probability measure on the mapping class group, for almost every random walk there is a uniquely ergodic measure foliation which is the limit of every orbit for the random walk in Teichm\"uller space. Moreover when the probability measure has finite entropy and finite first logarithmic moment with respect to the Teichm\"uller distance, the space of uniquely ergodic measured foliations is the Poisson boundary. 

Parallels between the mapping class group of a surface with its action on the Teichm\"uller space and the group $\Out(F_N)$ of outer automorphisms of a free group and the Outer Space $CV_N$, defined by Culler-Vogtmann \cite{CV86}, has been drawn and in many ways the study of surfaces has been a guideline for understanding $\Out(F_N)$ and its action on $CV_N$. However the geometry of $CV_N$ lacks a lot of nice features of the Teichm\"uller space and in particular does not have a natural symmetric metric. Folding paths in $CV_N$ are the closest to play the role of geodesic paths, however the nature of the forward direction of the folding path, i.e., folding, is unlike the backward direction, i.e., unfolding. This creates difficulty finding analogues for vertical and horizontal foliations for a quadratic differential and there is no symmetry between them. Finally, cocompactness in $CV_N$ seems to be weaker than it is in Teichm\"uller space. 
In particular it is easy to produce recurrent (after projecting to $CV_N/\Out(F_N)$) folding paths which are not {\em uniquely ergodic} in either direction and therefore the obvious analogue of Masur Criterion fails. 

In this article we study how non-unique ergodicity results in a transverse decomposition of the graphs in the folding path. Our results are similar to a theorem of McMullen, which shows that distinct ergodic measures on a singular foliation on a surface give rise to distinct components of the limiting noded surface in the Mumford compactification moduli space \cite{McM13}. One application of our results is that a folding, \emph{resp.} unfolding, ray with non-uniquely ergodic limit tree, \emph{resp.} lamination, is ``slow'' when projecting to the {\em free factor graph}.  To avoid technicalities, we state a result that is an easy consequence of our results, and we only consider folding.

\begin{thm}\label{distortion}
 Let $\pi:CV_N \to \mathcal{FF}$ be the projection map.  Suppose that $T_t$ is a folding path parameterized by arc length and with limit $T$.  If $T$ is non-uniquely ergodic, then $\lim_{t \to \infty} d(\pi(T_0),\pi(T_t))/t=0$.
\end{thm}

Recall that points in the boundary of Culler-Vogtmann Outer space are represented by very small actions of $\FN$ on trees that either are not free or else are not simplicial.  For such a tree $T$, $L(T)$ denotes the algebraic lamination associated to $T$.  Use $M_N \subseteq Curr(\FN)$ to denote currents that represent elements from the minset of the $Out(\FN)$-action on $\mathbb{P}Curr(\FN)$.  The tree $T$ is uniquely ergodic if $L(U)=L(T)$ implies that $U$ is homothetic to $T$.  The lamination $L(T)$ is uniquely ergodic if for $\eta, \eta' \in M_N$, if both are supported on $L(T)$, then  $\eta$ is homothetic to $\eta'$.   Let $\mathcal{UE}$ stand for the space of very small trees that are uniquely ergodic and which have uniquely ergodic lamination.  Let $\mathcal{FF}$ denote the complex of free factors, and let $x_0 \in \mathcal{FF}$ be some point.  We give two applications of Theorem \ref{distortion}:

\begin{thm}\label{limit set}
 Let $H \leq Out(F_N)$.  If the orbit map $H \to \mathcal{FF}:h  \mapsto hx_0$ is a quasi-isometric embedding, then $\text{Limit}(H)\subseteq \mathcal{UE}$.
\end{thm}

In the statement $\text{Limit}(H)$ denotes the accumulation points of $H$ in compactified Outer space.

The Poisson boundary for certain random walks on $Out(\FN)$ was recently described by Horbez \cite{Hor14b}.  We improve the result of Horbez via the following application of Theorem \ref{distortion} and recent work of Tiozzo and Maher-Tiozzo \cite{Tio14, MT14}.

\begin{thm}\label{random walks}
    Let $\mu$ be a non-elementary distribution of $Out(\FN)$ with finite first moment with respect to the Lipschitz metric on Outer space.  Then
  \begin{enumerate}
  \item [(i)] There exists a unique $\mu$-stationary probability measure $\nu$ on the space $\partial CV_N$, which is purely non-atomic and concentrated on $\mathcal{UE}$; the measure space $(\partial CV_N, \nu)$ is a $\mu$-boundary, and
  \item [(ii)] For almost every sample path $(w_n)$, $w_nx_0$ converges in $\overline{CV}_N$ to $w_\infty \in \mathcal{UE}$, and $\nu$ is the corresponding hitting measure.
  \item [(iii)] For almost every sample path $(w_n)$, there is a Lipschitz geodesic $T_t$ such that 
  \[
  \lim_{n \to \infty} d(w_nx_0, T_{Ln})/n =0
  \]
  \item [(iv)] The measure space $(\partial CV_N, \nu)$ is the Poisson boundary.
  \end{enumerate}
\end{thm}

The above applications use only a small part of our analysis of folding and unfolding paths in Outer space.  We will now outline our main result.  

\subsection{Folding/unfolding sequences and the decomposition theorem}

Our framework is that of folding/un-folding sequences; these objects live in the category of Stallings foldings, but we allow more than one fold at each step.  A folding/un-folding sequence comes with vector spaces of \emph{length measures} and \emph{width measures}, and a main point of our approach is to study their interaction.  

\begin{rem}
Every folding/un-folding sequence can be "filled-in" to get a \emph{liberal folding path} in the sense of \cite{BF14}.  In particular, every folding/un-folding sequence gives rise to  Lipschitz geodesics in Outer space.  Conversely, given a Lipschitz geodesic in Outer space, one can study the family of folding/un-folding paths arising from it.  In short, our analysis provides a method for analyzing the behavior of a geodesic in Outer space via the dynamical properties of its endpoints.  Analogous methods in Teichm\"uller space have been fruitful.
\end{rem}

Formally, a \emph{folding/unfolding sequence} is a sequence
\[
\ldots \to G_{-1} \to G_0 \to G_1 \to \ldots
\] 
of graphs together with maps $f_n:G_n \to G_{n+1}$ such that for any $m \leq n$, the composition
\[
f_n \circ f_{n -1} \circ \ldots \circ f_m:G_m \to G_{n+1}
\]
is a \emph{change of marking}, \emph{i.e.} a homotopy equivalence that sends every edge of $G_m$ to a reduced edge path.  We assume that folding/un-folding sequences are \emph{reduced}, which means that there there cannot be an invariant subgraph as $n \to \pm \infty$ in which no folding/unfolding occurs.  We assume that a marking between $G_0$ and fixed base point in Outer space has been specified, so a folding/unfolding sequence is a sequence of open simplices in Outer space.  

A {\em length measure} on $(G_n)$ is a sequence of vectors $(\vec\lambda_n)$, where $\vec\lambda_n$ is a vector that assigns a non-negative number, \emph{i.e.} a length, to every edge of $G_n$, in such a way that, when equipped with those lengths and for every $n$, the $f_n$ restricts to an isometry on every edge of $G_n$. Equivalently for every $n$
\[ \vec\lambda_n = A_n^T \vec\lambda_{n+1} \]
where $A_n$ is the incidence matrix of the graph map $f_n:G_n\to G_{n+1}$.  When equipped with a length measure with nonzero components, $G_n$ becomes a marked metric graph; therefore, we obtain a sequence in the unprojectivized Outer space. 

As $n\to\infty$, the universal covers of the graphs $(G_n)$ converge to a topological tree $T$. A length measure $(\vec\lambda_n)$ on $(G_n)$ induces a pseudo metric on $T$ which can be viewed as a non-atomic length measure on $T$ in sense of Paulin; see \cite{Gui00} or below. Collapsing the subsets of diameter zero gives a topological tree that represents a simplex in the boundary of Outer space, and specifying a length measure gives a convergent sequence in compactified Outer space.  We prove that the space of projectivized length measures on $(G_n)$ is linearly isomorphic to the space of projectivized non-atomic length measures on $T$ and is a finite dimensional linear simplex spanned by non-atomic ergodic length measures on $T$. We say $T$ or the sequence $(G_n)$ is {\em uniquely ergometric} if this simplex is degenerate.

In the opposite direction, we define a {\em width measure} (called a \emph{current} in the main body of the paper) on $(G_n)$ to be a sequence of vectors $(\vec\mu_n)$, where, again, each vector $\vec\mu_n$ assigns a non-negative number to every edge of $G_n$. However, for width measures, we require
\[ \vec\mu_{n+1}= A_n \vec \mu_n.\]
The space of projectivized width measures forms a finite dimensional linear simplex. In this case, the simplex naturally and linearly embeds in the space of projective {\em currents} for the free group. 

We prove that the set of currents corresponding to the width measures on $(G_n)$ is identified, via a linear isomorphism, with the set of currents supported the {\em legal lamination} $\Lambda$ of the sequence. The legal lamination $\Lambda$ can be described as the collection of parametrized bi-infinite reduced paths in $G_0$ which can be {\em lifted} to bi-infinite reduced paths in $G_n$ for every $n<0$, \emph{i.e.}, they are images of bi-infinite reduced paths in $G_n$. If tree $T$ in the compactification of Outer space is an accumulation point of the sequence $(G_n)$ equipped with a length measure and as $n\to-\infty$, then $\Lambda$ is contained in the {\em dual lamination} of $T$.  A current is supported on the dual lamination of $T$ if and only if $\langle T, \mu \rangle=0$, where $\langle \cdot,\cdot \rangle$ is the Kapovich-Lustig intersection function.  So the width measures naturally form a simplex in the space of projectivized currents, which is spanned by ergodic currents supported on the legal lamination $\Lambda$. We say the lamination $\Lambda$ or the sequence $(G_n)$ is {\em uniquely ergodic} if this simplex is degenerate.

Before stating our main results, we will briefly return to the case of surface theory in order to state a generalization of Masur Criterion, which essentially follows from ideas of Masur.  Assume $(g_t)_{t\ge0}$ is a Teichm\"uller geodesic ray in the Teichm\"uller space of a surface $S$ and that $(g_t)$ corresponds to a quadratic differential on $S$ with vertical foliation $\CF_v$ realized on $g_0$. Assume $\mu_1,\ldots,\mu_k$ are mutually singular ergodic transverse measures on $\CF_v$ and also assume there is a sequence $(g_{t_n})_n$ with $t_n\to\infty$ with $n$, so that the projections of $(g_{t_n})_n$ to the moduli space of $S$ converge to $\Sigma$, where $\Sigma$ is a noded Riemann surface in the Deligne-Mumford compactification of the moduli space of $S$. Masur's arguments can be applied to prove that there is a decomposition of $\Sigma$ into $k$ subsets whose closures only intersect in nodes of $\Sigma$ and with the property that for $i\in\{1,\ldots,k\}$ and $\mu_i$-a.e. leaf $l$ of $\CF_v$, the image of $l$ in $\Sigma$, under the composition of the Tichm\"uller map $g_0\to g_{t_n}$ and the approximating map $g_{t_n}\to\Sigma$ accumulate on the $i$-th subset of the decomposition. This implies, in particular, that the complement of nodes of $\Sigma$ has at least $k$ components, and one obtains the following theorem; see  \cite{McM13}.

\begin{thm}\label{mcmullen}
	Assume $(g_t)_{t\ge 0}$ is a Teichm\"uller geodesic ray whose projection to the moduli space accumulates on a noded Riemann surface $\Sigma$ with $\Sigma \setminus \{{\rm nodes}\}$ which has $k$ components. Then the vertical foliation for the quadratic differential associated to $(g_t)$ admits no more than $k$ mutually singular transverse measures.
\end{thm}

We now return to the case of a folding/unfolding sequence.  Our main result is a statement that is similar to Theorem \ref{mcmullen}.  By a {\em transverse decomposition} of a graph we mean a family of subgraphs whose edge sets are pairwise disjoint.  Note that moduli space of metric graphs of rank $\le N$ is compact, so when $(G_n)$ is equipped with a length measure, we can consider subsequential limits as  $n\to \pm \infty$.  

By what we stated earlier, every length measure $\lambda$ on $(G_n)$ is a linear combination of ergodic length measures. The {\em ergodic components} of $\lambda$ are the ones that appear in this combination with nonzero coefficient.  Recall that every leaf $\omega$ of the legal lamination is a bi-infinite reduced path in $G_0$ and for every $n<0$ can be lifted to a bi-infinite reduced path in $G_n$ which is denoted by $\omega_n$.  Our main result is:  

\noindent {\bf Theorem}: \emph{
	Assume that the unfolding sequence $(G_n)_{n \leq 0}$ has non-uniquely ergodic legal lamination $\Lambda$, and let  $\mu_1,\ldots,\mu_k$ be mutually singular ergodic measures.  Assume that $(G_n)$ is equipped with a length measure and that the subsequence $(G_{n_i})$ converges to a metric graph $G$ in the moduli space of metric graphs. Then there is a transverse decomposition $H^0, H^1,\ldots, H^k$ of $G$ with the property that for every $j\in\{1,\ldots,k\}$ and for $\mu^j$-a.e. leaf $\omega$ of $\Lambda$, the image of $\omega_{n_i}$ (the lift of $\omega$ to $G_{n_i}$) under the approximating map $G_{n_i}\to G$ accumulates inside $H^j$.}

\emph{
	Assume that the folding sequence $(G_n)_{n \geq 0}$ has non-uniquely ergodic limit tree $T$, and let $\lambda^1,\ldots,\lambda^k$ are the ergodic components of a length measure $\lambda$ on $T$. Assume that $(G_n)$ is equipped with a length measure giving rise to $\lambda$ and that the subsequence $(G_{n_i})$  converges in the moduli space of metric graphs to a metric graph $G$. Then there is a transverse decomposition $H^0, H^1,\ldots, H^k$ of $G$ with the property that for every $j\in\{1,\ldots,k\}$ and for $\lambda^j$-a.e. point $x\in G_0$ and $i$ sufficiently large, the image of $x$ under the composition of the map $G_0\to G_{n_i}$ and the approximating map $G_{n_i}\to G$ falls in $H^j$.
}

\begin{rem}
Coulbois and Hilion studied currents supported on the dual lamination of a free tree in the boundary of Outer space; our techniques and results provide a generalization of those of Coulbois-Hilion in \cite{CH13}.
\end{rem}

\subsection{Outline of the paper}

In Section 2 we collect background and terminology.  

In Section 3, we introduce our main object of study: folding/un-folding sequences.  Length measures and currents on a folding/un-folding sequence are introduced.  We introduce a restriction that will be in place during the sequel: folding sequences are reduced--meaning that there cannot be an invariant subgraph for all time in which no folding/un-folding occurs.  

In Section 4 we focus on un-folding paths and define the legal lamination.  We pause briefly to discuss the implications of the reduced hypothesis.  Our first main result is Theorem \ref{spaces of currents are isomorphic}, which gives that the combinatorially-defined space of currents on the un-folding sequence is isomorphic to the space of currents supported on the legal lamination.  We then obtain Theorem \ref{transverse decomposition}, which ensures that distinct ergodic currents give rise to transverse decomposition of the graphs in the folding sequence that is almost invariant near minus infinity.  Next we focus on unfolding sequences whose underlying graphs converge in the compactified moduli space of graphs.  Our next main result is Theorem \ref{frequency in generic leaves}, which ensures that homotopically non-trivial parts of our transverse decomposition survive even if the sequence of graphs degenerates.  Finally, we focus on the case where the unfolding sequence is recurrent to the thick part of moduli; in this case we obtain a much better bound on the number of ergodic measures in Theorem \ref{recurrent implies bounded number of measures}.

In Section 5 we repeat the analysis of Section 4, but now focus on folding sequences.  The differences between the two scenarios are subtle but significant.  

In Section 6, we apply our technical results from Sections 4 and 5 to study projections to the complex of free factors.  Our main results are Theorems \ref{slow progress in FF for unfoldings} and \ref{slow progress in FF for foldings}, which ensure that a folding, \emph{resp.} un-folding, sequence with non-uniquely ergodic limit tree, \emph{resp.} legal lamination, makes slow progress when projected to the complex of free factors.  

In Section 7, we gather the required facts from the literature to apply our main results from Section 6 to settle Theorems \ref{limit set} and \ref{random walks}

\noindent \emph{Acknowledgements:}  We wish to thank Jon Chaika for a pleasant and enlightening discussion at the beginning of this project.  We also wish to thank Lewis Bowen, Thierry Coulbois, Arnaud Hilion,  Ilya Kapovich, Joseph Maher, Amir Mohammadi, and Giulio Tiozzo for informative conversations and comments.   We would also like to thank Jing Tao for pointing out to us the paper of McMullen.

\section{Preliminaries} \label{sec: preliminaries}

The reader can find background information about Outer space, the complex of free factors, and related issues at the beginning of Section \ref{applications}.

\subsection{Paths in graphs}\label{subsec: paths in graphs}
Given an oriented graph $G$, we define $\closure\Omega(G)$ to denote the set of reduced labeled directed paths in $G$; an element $\omega$ of $\closure\Omega(G)$ is an immersion from an interval $I$ of the form $[a,b], [a,\infty), (-\infty,b],$ or $(-\infty,\infty)$, with $a,b\in\BZ$, that sends $[i,i+1] \cap I$  homeomorphically to an edge $e_i$ of $G$. When $G$ is equipped with a length function, we assume the restriction of $\omega$ to $[i,i+1]$ has constant speed. When $\omega$ is defined on an interval $[a,b]$, the {\em simplicial length} of $\omega$ is $b-a$, denoted $|\omega|$; in this case, we say that $\omega$ is a \emph{finite labeled path}.

Given $a<b$ in $\BZ\cap I$, we use the notation $\omega[a,b]$ to denote the restriction of $\omega$ to $[a,b]$. We also define $\closure\Omega_\infty(G) \subset \closure\Omega(G)$ to denote the subset consisting of bi-infinite labeled paths. 

The {\em shift map} $S :\closure\Omega(G)\to\closure\Omega(G)$ is defined as follows: if $\omega: I\to G$, then  $S\omega(i)=\omega(i+1)$. The shift map restricts to a map $\closure\Omega_\infty(G) \to \overline\Omega_\infty(G)$. The quotient of $\closure \Omega(G)$, respectively $\closure\Omega_\infty(G)$, under the action $S$ is denoted by $\Omega(G)$, respectively $\Omega_\infty(G)$. Given $\gamma\in\Omega(G)$, a {\em labeling} of $\gamma$ is a an element of $\closure\Omega(G)$ that projects to $\gamma$ and has 0 in its domain. A special subset of $\Omega(G)$ is the set $EG$ of directed edges of $G$. 

A homotopy equivalence $G \to R_N$, where $R_N=\wedge_{j=1}^N S^1$, gives an identification of $\pi_1(G)$ with $F_N$; up to action of $Out(F_N)$ all homotopy equivalences come from choosing a maximal tree $K \subseteq G$, mapping $K$ to the vertex of $R_N$, and choosing a bijection from edges in $G \smallsetminus K$ to edges of $R_N$.  Choosing a lift of $K$ to the universal cover $\cover G$ of $G$ and requiring $\omega(0) \in K$ gives an embedding of $\closure\Omega_\infty(G)$ as a compact subset of $\D^2 F_N$, the {\em double boundary} of $F_N$; also we have a natural surjective map $\D^2 F_N\to \Omega_\infty(G)$.

A finite directed path $\gamma\in\Omega(G)$ gives a subset ${\rm Cyl}(\gamma) \subseteq \closure\Omega_\infty(G)$, which consists of all $\omega\in\closure\Omega_\infty(G)$, with the property that $\omega[0,|\gamma|]$ projects to $\gamma$ and with the same orientation.

\subsection{Currents}\label{subsec: currents}
A {\em current} $\mu$ on $F_N$ is an $F_N$-invariant and \emph{flip-invariant} positive Borel measure on $\D^2F_N$ that takes finite values on compact subsets.  Here, flip-invariant means invariant under the natural involution of $\D^2 F_N$.  

Given a graph $G$ with $\pi_1(G)$ identified with $F_N$, using a natural embedding of $\closure\Omega_\infty(G)$ in $\D^2 F_N$ mentioned earlier, a current $\mu$ on $F_N$ induces a finite measure on $\closure\Omega_\infty(G)$. Since $\mu$ is $F_N$ invariant, it is easy to see that this measure on $\closure\Omega_\infty(G)$ does not depend on which natural embedding of $\closure\Omega_\infty(G)$ in $\D^2 F_N$ was used.
The measure on $\closure\Omega_\infty(G)$, which is still denoted by $\mu$, is invariant under the shift map. In fact, the set of currents on $F_N$ is naturally identified with the set of shift-invariant finite measures on $\closure\Omega_\infty(G)$.

Given a finite path $\gamma\in\Omega(G)$, define $\mu(\gamma):=\mu({\rm Cyl}(\gamma))$. In particular, for every directed edge $e$ of $G$, we have the weight $\mu_G(e)$. Regard the vector $\vec\mu_G \in \mathbb{R}^{|EG|}$ as a column vector: $\vec\mu_G(e)$, the component associated to $e\in EG$, is equal to $\mu_G(e)$.

The set of numbers $\mu(\gamma)$ for all $\gamma\in\Omega(G)$ provide a coordinate system for the space of currents; indeed, the topology on the space of currents is the product topology corresponding to finite $\gamma \in \Omega(G)$ (See \cite{Kap06}) In particular, these coordinates $(\mu(\gamma))_{\gamma\in\Omega(G)}$ uniquely identify $\mu$. Moreover, given a collection $(\alpha(\gamma))_{\gamma\in\Omega(G)}$ of non-negative numbers, there is a unique current $\mu$ on $\pi_1(G)$ with $\mu(\gamma)=\alpha(\gamma)$ for every $\gamma\in\Omega(G)$ if and only if given any path $\gamma$ in $\Omega(G)$, if $\gamma_1,\ldots,\gamma_k$ are all possible extensions of $\gamma$ to a longer path by adding a single edge at the end of $\gamma$, then the Kolmogorov extension property holds:
\[ \alpha(\gamma) = \sum_{i=1}^k \alpha(\gamma_i).\]

\subsection{Length vectors}\label{subsec: length vectors}
A {\em length vector} on a graph $G$ is a vector $\vec \lambda_G$ in $\BR^{|EG|}$ whose component associated to every edge $e\in EG$, denoted by $\lambda_G(e)$, is a nonnegative real number. If all the components of $\vec\lambda_G$ are positive, the length vector induces a metric on $G$ which is obtained by identifying every edge $e\in EG$ with a closed interval of length $\lambda_G(e)$ and defining the distance between two points to be the length of the shortest path connecting them. In general, $\vec\lambda_G$ induces a pseudo-metric where the distance of two distinct points is allowed to be zero. In this case, $\vec\lambda_G$ induces a metric on the graph which is obtained by collapsing all edges $e$ with $\lambda_G(e)=0$. 
Given a length vector $\vec\lambda_G$ on $G$, the length of a reduced path $\gamma$ in $G$ can be defined as the sum of lengths of its edges. 

A special example of a length vector on $G$ is the {\em simplicial length} $\vec 1_G$ where every component is equal to $1$. The induced metric on $G$ is also called the {\em simplicial metric}.

\subsection{Morphisms}\label{subsec: morphisms}
A {\em morphism} between graphs $G$ and $H$ is a function $f:G\to H$ that takes edges of $G$ to non-degenerate reduced edge paths in $H$. The {\em incidence matrix} $M_f$ of the morphism is a matrix whose columns are indexed by elements of $EG$ and its rows are indexed by elements of $EH$. Given $e\in EG$ and $e'\in EH$, $M(e',e)$, the $(e',e)$-entry of the matrix, is the number of occurrences of $e'$ in the reduced path $f(e)$.

We restrict to {\em change of marking} morphisms, where, in addition, $f$ is required to be a homotopy equivalence. A morphism is a {\em permutation} if the incidence matrix is a permutation matrix.

Given a length vector $\vec\lambda_H$ for $H$, the map $f$ naturally induces a length vector $\vec\lambda_G$ on $G$ and we have
\begin{equation}\label{eq: morphisms and length} 
\vec\lambda_G = M_f^T \vec\lambda_H.
\end{equation}

In other words, $\vec\lambda_G$ gives the $f$-pull-back of the metric coming from $\vec\lambda_H$ on $H$; with these metrics, $f$ is a locally isometric immersion on every edge of $G$.  By a morphism between two metric graphs, we always mean one which is a local isometry on every edge of the domain. We say a morphism $f$ {\em preserves simplicial length} if $\vec 1_G = M_f^T \vec 1_H$; equivalently the image of every edge of $G$ is a single edge of $H$.

We usually assume the vertices of graphs have valence at least $3$, but allowing graphs with valence $2$ vertices, one gets the Stalling factorization of a morphism $f:G\to H$ as a composition $f_2\circ f_1$ of morphisms $f_1:G\to G'$ and $f_2:G'\to H$, where $f_1$ is topologically a homeomorphism and $f_2$ maps every edge to a single edge, i.e., preserves simplicial length. We refer to the Stallings factorization as the {\em natural factorization} of the morphism $f$. 

A topologically one-to-one morphism $f:G\to H$, induces a one-to-one map from $\Omega(G)$ to $\Omega(H)$. We can use this to obtain a one-to-one map $\closure\Omega_\infty(G)\to\closure\Omega_\infty(H)$, which sends $\omega_G\in\closure\Omega_\infty(G)$ to an element $\omega_H\in\closure\Omega_\infty(H)$ which labels $f\circ \omega_G$ in the same direction and $\omega_H(0)=f(\omega_G(0))$; we denote this map by $f$ as well. When $f$ is a homeomorphism, a left inverse for the induced map $\closure\Omega_\infty(G)\to\closure\Omega_\infty(G)$ can be defined which sends $\omega_H$ to $\omega_G$, with the property that $\omega_G[0,1]$ is mapped to $\omega_H[a,b]$ with $a\le 0$ and $b\ge1$. More precisely, assume the $f$ image of $e\in EG$ is the path $e_0\cdots e_k$ in $H$. Then for every $i=0,\ldots,k$, $S^if$, which is the composition of $f:\closure\Omega_\infty(G)\to\closure\Omega_\infty(H)$ with the $i$-th iterate of the shift map, provides a one-to-one correspondence between ${\rm Cyl}_G(e)$ and ${\rm Cyl}_H(e_i)$. In particular when $\mu$ is a current on $\pi_1(G)=\pi_1(H)$,
\begin{equation}\label{eq: disjoint union 0}
	\mu(e) = \mu(e_i), \quad i=0,\ldots,k.
\end{equation}

Given a morphism $f:G\to H$, we say a path in $\Omega(G)$ is {\em legal} if its $f$-image is reduced, i.e. is an element of $\Omega(H)$. We use $\Omega_\infty^L(G)$ to denote the set of bi-infinite legal paths in $G$, and $\closure\Omega_\infty^L(G)$ to denote all possible labelings of elements of $\Omega_\infty^L(G)$. The $f$-image of these paths gives a subset of $\Omega_\infty(H)$, which we denote by $\Lambda_f$. The lift $\closure\Lambda_f$ of $\Lambda_f$ to $\closure\Omega_\infty(H)$, the set of all possible labelings of elements of $\Lambda_f$, is a compact subset of $\closure\Omega_\infty(H)$. 

Since $f$ induces a monomorphism on the level of fundamental groups, it induces a one-to-one map from $\D^2\pi_1(G)$ to $\D^2\pi_1(H)$ and therefore the map from $\Omega_\infty^L(G)$ to $\Lambda_f$ is one-to-one and onto. 
There is also an induced one-to-one map from $\closure\Omega_\infty^L(G)$ to $\closure\Lambda_f$, denoted by $f$, which sends $\omega_G$ to $\omega_H$, so that $\omega_G(0)$ is mapped to $\omega_H(0)$.

When $f:G\to H$ preserves simplicial length, the map $f:\closure\Omega_\infty^L(G)\to\closure\Lambda_f$ is invertible and for every $\omega_H$ in $\Lambda_f$ there is a unique $\omega_G$ whose image by $f$ is $\omega_H$. In particular for every $e_H\in EH$
\begin{equation}\label{eq: disjoint union 1} 
{\rm Cyl}_H(e_H)\cap\closure\Lambda_f = \biguplus_{f(e_G)=e_H} \left(f({\rm Cyl}_G(e_G))\cap\closure\Lambda_f\right),
\end{equation}
is a disjoint union over all edges $e_G\in EG$ which are mapped to $e_H$. Assume that we have a change of marking $f:G\to H$ and a current $\mu$ on $\pi_1(H)$. Recall that $\mu$ induces a shift-invariant finite measure on $\closure\Omega_\infty(H)$, as well as on $\closure\Omega_\infty(G)$ via the isomorphism $f_*:\pi_1(G)\to\pi_1(H)$, which we still call $\mu$. If $\mu$ is supported on $\closure\Lambda_f$, then it follows from \eqref{eq: disjoint union 1} that for every $e_H\in EH$
\begin{equation}\label{eq: disjoint union 2} 
\mu(e_H) = \sum_{\{ e\in EG\,:\, f(e)=e_H\}}\mu(e_H).
\end{equation}

When $f:G\to H$ is a change of marking morphism, we can use the natural factorization $f=f_2\circ f_1$ with $f_1:G\to G'$ a homeomorphism and $f_2:G'\to H$ preserving the simplicial length. Then $\Lambda_f = \Lambda_{f_2}$. For every $\omega_H\in\Lambda_f$, we find a unique $\omega_{G'}\in\closure\Omega_\infty(G')$ with $f_2(\omega_{G'})=\omega_H$. Then we use the left inverse of the map $f_1$ explained earlier to find $\omega_G\in\closure\Omega_\infty(G)$ so that $S^if(\omega_G) = \omega_H$ with $0\le i < |f(\omega_G[0,1])|$ and $|f(\omega_G[0,1])|$ the simplicial length of the $f$-image of the edge $\omega_G[0,1]$.

If $\mu$ is a current supported on $\closure\Lambda_f$ and $e_H\in EH$, by \eqref{eq: disjoint union 2}
\[ \mu(e_H) = \sum_{\{e'\in EG'\,:\, f_2(e')=e_H\}}\mu(e').\]
For every $e' \in EG'$, there is a unique $e\in EG$ so that $e'\subset f_1(e)$ as a sub-arc and by \eqref{eq: disjoint union 0} $\mu(e)=\mu(e')$. Hence
\begin{equation}\label{eq: disjoint union}
	\mu(e_H) = \sum_{\{e'\in EG'\,:\, f_2(e')=e_H\}}\mu(e') = \sum_{\{e\in EG\,:\, f(e)\supset e_H\}} \mu(e).
\end{equation}

As a result:

\begin{lem}\label{morphisms and currents}
	Given a change of marking morphism $f:G\to H$ with incidence matrix $M_f$. If $\mu$ is a current supported on $\Lambda_f$ and $\vec\mu_G$ and $\vec\mu_H$ are respectively the corresponding vectors for $G$ and $H$, then
\begin{equation}\label{eq: morphisms and currents}
	\vec\mu_H = M_f \vec\mu_G.
\end{equation}
\end{lem}

\section{Folding/Unfolding Sequences}\label{sec: folding/unfolding}

A {\em folding/unfolding sequence} is a sequence of graphs $(G_n)_{a\le n\le b}$, $a,b\in\BZ\cup\{\pm\infty\}$, with change of marking morphisms $f_{k,l}:G_k\to G_l$ for $k<l$ such that if $k<l<m$, $f_{k,m}=f_{l,m} \circ f_{k,l}$.  We use $f_n$ to denote $f_{n,n+1}$. Special cases are folding, respectively unfolding, sequences where $0\le n <\infty$, respectively $-\infty<n\le 0$.

An {\em invariant sequence of subgraphs} is a sequence of non-degenerate proper subgraphs $E_n\subset G_n$ with the property that $f_n$ restricts to a change of marking morphism $E_n\to E_{n+1}$ for every $a\le n < b$. We say such a sequence is a {\em stabilized sequence of subgraphs} if the restriction of $f_n$ to $E_n\to E_{n+1}$ is a permutation for $n$ sufficiently large. A sequence is {\em reduced} if it admits no stabilized sequence of subgraphs. 



\subsection{Length measures}\label{subsec: length measures}
A {\em length measure} for the folding/unfolding sequence $(G_n)_{a\le n \le b}$ is a sequence of length vectors $(\vec \lambda_n)_{a\le n\le b}$ where $\vec\lambda_n$ is a length vector on $G_n$ and for $a\le n <b$
\[ \vec\lambda_{n} = M_n^T\vec\lambda_{n+1}\]
with $M_n = M_{f_n}$ the incidence matrix of $f_n$. When $b<\infty$ is finite, a length vector on $G_b$ provides a length measure on the sequence $(G_n)$. We define the {\em simplicial length measure} to be the length measure induced by the simplicial length vector $\vec 1_{G_b}$ on $G_b$.  

Given a length measure $(\vec\lambda_n)$ on $(G_n)$ and identifying $\pi_1(G_n)$ with $F_N$ for one of the graphs $G_n$, we can use the markings obtained form the maps $f_n$ and realize $(G_n,\lambda_n)$, the graph $G_n$ equipped with the metric induced by $\vec\lambda_n$, as a marked metric graph in $cv_N$. Therefore we obtain a sequence in $cv_N$ and its projectin $CV_N$.

The set of sequences of vectors $(\vec v_n)_{a\le n \le b}$ with $v_n\in \BR^{|EG_n|}$ and with
\[ \vec v_n = M_n^T \vec v_{n+1}\]
for every $a\le n <b$ is naturally a finite dimensional real vector space, whose dimension is bounded by $\liminf_{n\to\infty} |EG_n|$. The set of length measures on $(G_n)_{a\le n\le b}$, denoted $\CD((G_n)_n)$,  is the cone of non-negative vectors in this space.


\subsection{Currents on sequences and area}\label{subsec: currents on sequences}
A {\em current} for the folding/unfolding sequence $(G_n)_{a\le n\le b}$ is a sequence of non-negative vectors $(\vec\mu_n)_{a\le n\le b}$ where each $\vec\mu_n$ is a non-negative vector in $\BR^{|EG_n|}$ and for $a\le n < b$
\[ \vec\mu_{n+1}=M_n\vec\mu_n.\]

When $a>-\infty$ is finite, currents for the sequence $(G_n)$ are completely determined by non-negative vectors in $\BR^{|EG_a|}$. In particular, the {\em frequency current} is defined by starting from $\vec\mu_a = \vec 1_{G_a}$ and defining
$\vec\mu_{n+1} = M_n\vec\mu_n$ for $a<n\le b$.

Similar to the case of length measures, we can consider the set $\CC((G_n)_n)$ of  currents on the sequence $(G_n)_{a\le n\le b}$ as the cone of non-negative vectors in the finite dimensional vector space whose vectors are sequences $(\vec v_n)_{a\le n\le b}$ with $\vec v_n\in\BR^{|EG_n|}$ for every $a\le n\le b$ and
\[ \vec v_{n+1} = M_n \vec v_n\]
for every $a\le n < b$. The dimension of this space evidently is bounded by $\liminf_{n\to -\infty}|EG_n|$.

Given a length measure $(\vec\lambda_n)_n$ and a current $(\vec\mu_n)_n$ for the sequence $(G_n)_n$, we define the {\em area} to be the scalar
\[ A = \vec\mu_n^T\vec\lambda_n.\]
Area is independent of $n$, and for reduced sequences:

\begin{lem}\label{length and frequency decay}
	Given a reduced folding/unfolding sequence $(G_n)_{a\le n\le b}$ for every length measure $(\vec\lambda_n)_n$ on $(G_n)_n$,
	\[ \lim_{n\to\infty} \lambda_n(e_n) = 0 \quad {\rm and} \quad \lim_{n\to-\infty} \lambda_n(e_n) = \infty\]
	when $e_n\in EG_n$ is chosen arbitrarily. Similarly if $(\vec\mu_n)_n$ is a current on $(G_n)_n$
	\[ \lim_{n\to\infty} \mu_n(e_n) = \infty \quad {\rm and} \quad \lim_{n\to-\infty} \mu_n(e_n) = 0\]
	for an arbitrary choice of $e_n\in EG_n$ and the limits are considered only when $a=-\infty$ or $b=\infty$.
\end{lem}


\section{Unfolding sequences}\label{sec:legal sequence}

Recall that an {\em unfolding sequence} $(G_n)_{n\le0}$ is given by
\begin{equation}\label{eq: unfolding sequence}
\xymatrix{
\cdots \ar[r]^{f_{n-1}} & G_{n} \ar[r]^{f_{n}} \ar[r] & \cdots \ar[r]^{f_{-1}} & G_0
}
\end{equation}
with each $f_n, n<0,$ a change of marking morphism. For every $n<0$, we define $\phi_n = f_{-1}\circ\cdots\circ f_n:G_n\to G_0$. We continue with the assumption that $(G_n)_n$ is a reduced sequence. Also we assume that $(\vec\lambda_n)_n$ is the simplicial length measure, i.e., $\vec\lambda_0 = \vec 1_{G_0}$ and for every $n<0$
\[ \vec\lambda_n = M_{n}^T \vec\lambda_{n+1}.\]
Equivalently the component $\lambda_n(e)$ of $\vec\lambda_n$ associated to $e\in EG_n$ is the simplicial length of $\phi_n(e)$ in $G_0$. 





\subsection{Legal lamination}\label{subsec: legal lamination}

Given an unfolding sequence as in \eqref{eq: unfolding sequence}, we say a path in $\Omega(G_n)$ is {\em legal} if it is legal with respect to the morphism $\phi_n$, i.e., its image under $\phi_n$ is a reduced path. Every labeling of such a path in $\closure\Omega(G_n)$ is also {\em legal}. As before, we use the notations $\Omega^L_\infty(G_n)$ and $\closure\Omega^L_\infty(G_n)$ to denote the set of bi-infinite legal paths in $\closure\Omega_\infty(G_n)$. 

We define
\[\Lambda = \bigcap_n \Lambda_{\phi_n} = \bigcap_n \phi_n(\Omega^L_\infty(G_n)).\] 
The {\em legal lamination} $\closure\Lambda$ of the sequence is the subset of $\closure\Omega_\infty(G_0)$ that consists of all labellings of elements of $\Lambda$.
An element of $\closure\Lambda$ is called a {\em leaf} of the lamination. On can also consider the pre-image of $\Lambda$ in the double boundary $\D^2\pi_1(G_0)$ which will be a $\pi_1(G_0)$-invariant closed subset of the double boundary and is a lamination in the conventional sense of \cite{CHL08a}.

From our discussion of morphisms in \S \ref{subsec: morphisms}, it follows that since every leaf of $\closure\Lambda$ is image of a legal bi-infnite path in $G_n$:

\begin{lem}\label{pre-image of leaves}
    Suppose $(G_n)_{n\le0}$ is an unfolding sequence with legal lamination $\closure\Lambda$. For every leaf $\omega$ of $\closure\Lambda$ and $n$, there is a unique $\omega_n\in\closure\Omega_\infty(G_n)$ with the property that $\phi_n(\omega_n)$ is a relabeling of $\omega$ (in the same direction) and $\phi_n(\omega_n[0,1]) = \omega[a,b]$ for integers $a\le 0 < b$.
\end{lem}

\subsection{Legal laminations of reduced unfolding sequences}\label{subsec: reduced legal lamination}

We pause to give an informal discussion about laminations that arise as as the legal lamination of a reduced unfolding sequence.  Although the hypothesis of being reduced may seem weak, it is strong enough to rule out certain pathologies that occur in general algebraic laminations on $F_N$. 

Suppose that $T$ is a very small tree with a non-abelian point stabilizer $H$, and let $L(T)$ denote the dual lamination of $T$.  Then the subshift $L_v(T) \subseteq E^{\pm}R_N^\BZ$ of the full shift on the oriented edges of $R_N$, $L_v(T)=\{x\in \D^2 F_N|$ the geodesic in $\cover R_N$ corresponding to $x$ contains $v\}$, $v \in \cover R_N$ some fixed vertex, has positive entropy, and the space of currents supported on $L(T)$ is infinite dimensional; this is because $L(T)$ contains $\D^2 H$.  The main results of this paper involve obtaining finiteness results about the space of currents supported on the legal lamination $\Lambda$ of a reduced unfolding sequence $(G_n)_n$.  As a warm-up for these arguments we bring:

\begin{prop}
 If $\Lambda$ is the legal lamination of a reduced folding sequence $(G_n)_{n \leq 0}$, then $\Lambda$ has zero entropy.
\end{prop}

For the proof we consider $\Lambda$ as a subshift of the full shift on the oriented edges of $G_0$; recall that for a subshift $X$ of $\mathcal{A}^\BZ$, one has $h(X)=\lim \log |B_k|/k$, where $B_k$ is the set of length-$k$ strings that occur in elements of $X$.  

\begin{proof}
 For $l$ fixed, Lemma \ref{length and frequency decay} gives that for $n$ sufficiently large, every loop in $G_n$ has length at least $l$.  Notice that if a legal path $\gamma \in \Omega(G_n)$ has length at least $2||\vec \lambda_n||_1$, $\gamma$ must contain a loop, hence if $\gamma$ has length $p$, then the frequency with which  $\gamma$ visits a vertex of $G_n$ is certainly bounded by $(2N-2)p/l$.  Define $L_n$ to be the total number of legal turns in $G_n$; then $L_n \leq N(N-1)-1$ for all $n$ ($G_n$ must have an illegal turn, else $f_{n-j,n}$ is the identity for $j>0$).  
 
 Now assume that $l$ is large, that every loop in $G_n$ has length at least $l$, and that $m$ is large compared to $l$.  Then the graph $G_n$ ensures that we have the estimate \[|B_m|\leq ||\lambda_n||_1 C^{\frac{(2N-2)L_n}{l}m}\] where $C\leq 2N-1$ certainly holds.  It follows immediately that $h(\Lambda)=0$.  
\end{proof}

We also mention the following:

\begin{prop}\label{minimal components}
 $\Lambda$ contains at most $3N-3$ minimal sublaminations.
\end{prop}

\begin{proof}
 If $l_0, l_1 \in \Lambda$ satisfy that $\overline{\{S^il_0|i\in\BZ\}} \neq \overline{\{S^il_1|i\in\BZ\}}$, then for all sufficiently large $n$, there must be $e_n^0, e_n^1\in EG_n$, such that $l_i$ crosses $e_n^i$ and does not cross $e_n^{1-i}$.
 \end{proof}
 
 The bound in Proposition \ref{minimal components} certainly is not sharp, as will become clear in the sequel.
 
\subsection{Currents on the legal lamination}\label{subsec: currents on legal lamination}
Recall that the set of currents on $\pi_1(G_0)$ is naturally identified with the set of finite shift-invariant measures on $\closure\Omega_\infty(G_0)$. In the presence of such a correspondence, by a {\em probability current} we mean one which induces a shift-invariant probability measure on $\closure\Omega_\infty(G_0)$. Note that since
\[ \closure \Omega_\infty(G_0) = \biguplus_{e\in EG_0} {\rm Cyl}(e),\]
equivalently
\[ \sum_{e\in EG_0} \mu(e) =1,\]
where $\mu(e) = \mu({\rm Cyl}(e))$.  

We also use this correspondence to define $\CC(\Lambda)$, {\em the set of currents supported on $\closure\Lambda$}, to denote those currents whose support in $\closure\Omega_\infty(G_0)$ is a subset of $\closure\Lambda$. The set of currents for $G_0$ is identified by the set of currents for $G_n$ via the identification of $\pi_1(G_n)$ and $\pi_1(G_0)$ provided by $\phi_n$. Hence for $\mu\in\CC(\Lambda)$ and every $n\le 0$, we obtain a current $\mu_n$ for $G_n$. Obviously $\mu_n$ is supported on the set of bi-infinite paths which are legal with respect to $\phi_n$. We claim that $\CC(\Lambda)$ is isomorphic to the space $\CC((G_n)_n)$ of currents on the unfolding sequence $(G_n)_n$. 

\begin{thm}\label{spaces of currents are isomorphic}
    Given a reduced unfolding sequence $(G_n)_{n\le0}$ with legal lamination $\closure\Lambda$, there is a natural linear isomorphism between $\CC((G_n)_n)$ and $\CC(\Lambda)$. 
\end{thm}
\begin{proof}
    Assume $(\vec\mu_n)_n\in\CC((G_n)_n)$ is a current on the unfolding sequence $(G_n)_n$. We first show how this sequence induces a current supported on $\closure\Lambda$. Given a finite directed labeled path $\omega\in\closure\Omega(G_0)$ and a finite directed path $\gamma\in\Omega(G_0)$, we use the notation
\[ \langle  \gamma , \omega \rangle \] 
to show the number of copies of $\gamma$ in $\omega$ which agree with the orientation of $\omega$. In particular for $n\le 0$ and $e\in EG_n$, we can consider $\langle \gamma, \phi_n(e)\rangle$ to denote the number of occurrences of $\gamma$ in $\phi_n(e_n)$ considered as a directed path in $G_0$. We define
\[ \alpha_n(\gamma) = \sum_{e\in EG_n}\mu_n(e) \langle \gamma,\phi_n(e)\rangle. \]
We claim for a fixed $\gamma$, $\alpha_0(\gamma)\le \alpha_{-1}(\gamma)\le\cdots\le\alpha_n(\gamma)\le\cdots$ is a monotonic increasing sequence. To see this assume $e\in EG_n$ and $e'\in EG_{n+1}$ are given. The number of occurrences of $e'$ in $f_n(e)$ as a directed sub-path is given by $M_n(e',e)$, the entry of incidence matrix of $f_n$ corresponding to $e'$ and $e$. Every occurrence of $\gamma$ as a directed sub-path of $\phi_{n+1}(e')$ produces $M_n(e',e)$ occurrences of $\gamma$ as a directed sub-path of $\phi_n(e)$. Hence
\begin{align*}
	\alpha_n(\gamma) &= \sum_{e\in EG_n}\mu_n(e) \langle \gamma, \phi_n(e) \rangle \\
		& \ge \sum_{e\in EG_n} \mu_n(e) \sum_{e'\in EG_{n+1}} M_n(e',e) \langle\gamma,\phi_{n+1}(e') \rangle \\
		&= \sum_{e'\in EG_{n+1}} \langle\gamma, \phi_{n+1}(e') \rangle \sum_{e\in EG_n} M_n(e',e)\mu_n(e) \\
		& = \sum_{e'\in EG_{n+1}} \langle\gamma, \phi_{n+1}(e') \rangle\, \mu_{n+1}(e') \\
		& = \alpha_{n+1}(\gamma).
\end{align*}
Obviously for every $e\in EG_n$, $\langle \gamma, \phi_n(e) \rangle$ is bounded from above by $\lambda_n(e)$, the simplicial length of $\phi_n(e)$. This implies that $\alpha_n(\gamma)$ is bounded for every $n$ by the area of the unfolding sequence with respect to the current $(\vec\mu_n)_n$ and the simplicial length $(\vec\lambda_n)_n$. Therefore $\mu(\gamma) = \lim_n \alpha_n(\gamma)$ exists. 

On the other hand if $\gamma_1,\ldots,\gamma_k\in \Omega(G_0)$ are all possible extensions of $\gamma$ in $G_0$ by adding a single edge, then we can see that every occurrence of $\gamma$ as a directed sub-path of $\phi_n(e)$ either extends to an occurrence of $\gamma_i$ as a sub-path of $\phi_n(e)$ for some $i=1,\ldots, k$ or $\gamma$ is the terminal sub-path of $\phi_n(e)$ and $\phi_n(e)$ does not extend beyond $\gamma$. The latter possibility contributes at most one instance of occurrence of $\gamma$ as a sub-path of $\phi_n(e)$ and therefore
\[\sum_{i=1}^k \langle \gamma_i, \phi_n(e) \rangle \le \langle \gamma,\phi_n(e) \rangle \le \sum_{i=1}^k \langle \gamma_i, \phi_n(e) \rangle + 1;\]
multiplying by $\mu_n(e)$ and adding for all the edges of $G_n$
	\[ \sum_{i=1}^k \alpha_n(\gamma_i) = \sum_{i=1}^k \sum_{e\in EG_n} \mu_n(e)\langle \gamma_i,\phi_n(e)\rangle  \le \alpha_n(\gamma) \le \sum_{i=1}^k \alpha_n(\gamma_i) + \sum_{e\in EG_n} \mu_n(e).\]
By lemma \ref{length and frequency decay} $\mu_n(e)$ tends to zero for edges of $G_n$ as $n\to\infty$ and therefore in the limit
\[ \mu(\gamma) = \sum_{i=1}^k \mu(\gamma_i).\]

Using the characterization of currents via a coordinate system described in \S \ref{subsec: currents}, we conclude that $\mu$ is a current on $\pi_1(G_0)$. Now we show that $\mu$ is supported on $\closure\Lambda$. Assume $n$ is fixed and $\gamma$ is a path in $\Omega(G_0)$ whose simplicial length is bigger than $2\lambda_n(e)$, twice the simplicial length of $\phi_n(e)$, for every $e\in EG_n$. We claim that $\mu(\gamma)>0$ only when for some $e\in EG_n$, $\gamma$ contains $\phi_n(e)$ as a subpath. This will imply that every leaf in the support of $\mu$ is a limit of $\phi_n(e_n)$ for a sequence of edges $e_n\in EG_n$ and therefore is contained in $\closure\Lambda$. If $\mu(\gamma)>0$ then for every $m$ sufficiently large, there is $e'\in EG_m$ with $\langle \gamma, \phi_m(e') \rangle >0$. But $\phi_m(e')$ is a concatenation of subpaths each of which is the $\phi_n$-image of an edge of $G_n$. Since $\gamma$ is a subpath of $\phi_m(e')$ and its simplicial length is larger than the twice the simplicial length of $\phi_n(e)$ for every $e\in EG_n$, $\phi_n(e)$ must be a subpath of $\gamma$ for some $e\in EG_n$. This proves the claim and we have constructed a natural map $\CC((G_n)_n)\to\CC(\Lambda)$.

Now suppose $\mu\in\CC(\Lambda)$ is given. For every $n\le 0$, we define a vector $\vec \mu_n$ in $\BR_{\ge 0}^{|EG_n|}$ where $\mu_n(e)$, the component of $\vec\mu_n$ associated to $e\in EG_n$, is given by $\mu_n(e) = \mu_n({\rm Cyl}(e))$. It is a consequence of lemma \ref{morphisms and currents} that
\[ \vec\mu_{n+1} = M_{n}\vec\mu_n, \quad n< 0.\]
Hence the sequence of vectors $(\vec\mu_n)_n$ is a current for the unfolding sequence $(G_n)_n$ and $(\vec\mu_n)_n \in \CC((G_n)_n)$ and we have a map $\CC(\Lambda)\to \CC((G_n)_n)$.

To prove that the two maps $\CC((G_n)_n)\to\CC(\Lambda)$ and $\CC(\Lambda)\to\CC((G_n)_n)$ are inverses of each other, we just need the next proposition.

\begin{prop}\label{legal current determined by sequence}
	Suppose $(G_n)_{n\ge 0}$ is a reduced unfolding sequence with legal lamination $\closure\Lambda$ and $\mu$ is a current supported on $\closure\Lambda$. Then for every path $\gamma\in\Omega(G_0)$
	\[ \mu(\gamma)= \mu({\rm Cyl}(\gamma)) = \lim_n \sum_{e\in EG_n} \mu_n(e) \langle \gamma, \phi_n(e) \rangle.\]
\end{prop}
\begin{proof}
    For the morphism $\phi_n:G_n\to G_0$, we use the natural factorization and factor it as a composition $\psi_n\circ\xi_n$ with $\xi_n: G_n\to G_n'$ a homeomorphism and $\psi_n:G_n'\to G_0$ a morphism that preserves the simplicial length. Then we can use a decomposition similar to \eqref{eq: disjoint union 1}
    \[ {\rm Cyl}_{G_0}(\gamma) \cap \closure\Lambda = \biguplus_{\psi_n(\gamma')=\gamma} (\psi_n({\rm Cyl}_{G_n'}(\gamma')) \cap \closure\Lambda),\]
    where the disjoint union is over all directed paths $\gamma'$ in $G_n'$ whose image is $\gamma$. As a result
    \[ \mu(\gamma) = \sum_{\psi_n(\gamma')=\gamma}\mu_{G_n'}(\gamma') \]
    Where $\mu_{G_n'}(\gamma')$ is the $\mu$-measure of ${\rm Cyl}_{G_n'}(\gamma')$. 
    When the $\xi_n$-pre-image of $\gamma'$ is contained in an edge $e\in EG_n$ then the composition of $\xi_n$ with an iterate of the shift map on $\closure\Omega_\infty(G_n')$ identifies ${\rm Cyl}_{G_n}(e)$ with ${\rm Cyl}_{G_n'}(\gamma')$ and as a result $\mu_n(e) = \mu_{G_n'}(\gamma')$; there are $\langle \gamma, \mu \rangle$ such paths $\gamma'$ and we conclude
    \[ \mu(\gamma) \ge \sum_{e\in EG_n}\mu_n(e) \langle \gamma, \phi_n(e) \rangle.\]
    Moreover the difference is the sum of $\mu_{G_n'}(\gamma')$ for all paths $\gamma'$ in $G_n'$ with $\psi_n(\gamma')=\gamma$ but so that the $\xi_n$-pre-image of $\gamma'$ in $G_n$ has a vertex of $G_n$ in its interior. The number of such paths $\gamma'$ is bounded by $(|\gamma|-1)(3N-3)^2$. Also ${\rm Cyl}_{G_n'}(\gamma')$ will be contained in the $\xi_n$-image of a cylinder ${\rm Cyl}_{G_n}(e)$ for an edge $e\in EG_n$ and therefore $\mu_{G_n'}(\gamma') \le \mu_n(e)$; hence
    \[ \mu(\gamma) - \sum_{e\in EG_n}\mu_n(e) \langle \gamma, \phi_n(e) \rangle \le (|\gamma|-1)(3N-3)^2 \left(\max_{e\in EG_n}\mu_n(e)\right).\]
    However by lemma \ref{length and frequency decay} we know that $\max_{e\in EG_n}\mu_n(e) \to 0$ as $n\to-\infty$ and this proves the proposition.
\end{proof}

Therefore we have provided a bijection between $\CC(\Lambda)$ and $\CC((G_n)_n)$. The fact that this is a linear isomorphism is easy.
\end{proof}

We obtain:

\begin{cor}\label{dimension of space of currents}
    The dimension of the space $\CC(\Lambda)$ of currents on $\Lambda$ is bounded by $\liminf_n |EG_n|$, and, in particular, is bounded by $3N-3$.
\end{cor}

\subsection{Ergodic currents}\label{subsec: ergodic currents}
As before assume $(G_n)_{n\le0}$ is a reduced unfolding sequence with legal lamination $\closure\Lambda$. A current $\mu\in\CC(\Lambda)$ is {\em ergodic} if whenever $\mu= c_1 \mu^1 + c_2\mu^2$ for $\mu^1,\mu^2\in\CC(\Lambda)$ then $\mu^1$ and $\mu^2$ are homothetic to $\mu$. Since $\CC(\Lambda)$ is finite dimensional there are at most finitely many non-homothetic ergodic currents on $\closure\Lambda$; this is because $\CC(\Lambda)$ is a simplex, a point that is explained below.

Assume $\mu^1,\ldots,\mu^k$ are the mutually singular ergodic probability currents on $\closure\Lambda$. Recall that given a finite directed labeled path $\omega\in\closure\Omega(G_0)$ and a finite directed (unlabeled) path $\gamma\in\Omega(G_0)$, we use the notation
\[ \langle  \gamma , \omega \rangle \] 
to show the number of copies of $\gamma$ in $\omega$ and in the same direction. To be more precise let $\chi_\gamma:\closure\Omega(G_0)\to\{0,1\}$ be the characteristic function for ${\rm Cyl}(\gamma)$. Then 
\[ \langle  \gamma , \omega \rangle = \sum_i \chi_\gamma(S^i\omega), \]
where $S$ is the shift map on $\closure\Omega(G_0)$. Obviously $\int \chi_\gamma d\mu = \mu(\gamma)$ for every current $\mu$.

By the Ergodic Theorem for the the shift invariant ergodic probability measure $\mu^i$ on $\closure\Lambda, i=1,\ldots,k$, the set of elements $\omega$ of $\closure\Lambda$ with the property that
\begin{equation}\label{eq: generic frequency}
\begin{aligned}
\mu^i(\gamma) = \int\chi_\gamma\, d\mu^i & = \lim_{n\to\infty}\frac1n\sum_{i=0}^{n-1}\chi_\gamma(S^i\omega) = 
\lim_{n \to \infty}
\frac{\langle \gamma, \omega[0,n] \rangle}
n \\
& = \lim_{n\to\infty}\frac1n\sum_{i=-n}^{-1}\chi_\gamma(S^i\omega)
= \lim_{n \to \infty} 
\frac{\langle \gamma, \omega[-n,0] \rangle}
n
\end{aligned}
\end{equation}
has full $\mu^i$-measure in $\closure\Lambda$. We can assume that a shift invariant subset of full $\mu^i$-measure in $\closure\Lambda$ has been chosen so that \eqref{eq: generic frequency} holds for every $\gamma$ in $\Omega(G_0)$. It will be also enough to continue by picking a path $\gamma$ with the property that $\mu^i(\gamma)\neq\gamma^j(\gamma)$ for every pair $i\neq j$ and use the above observation that \ref{eq: generic frequency} holds for this $\gamma$ and for $\omega$ generic with respect to $\mu^i$. In particular, we have that the set of generic leaves for $\mu^i$ is disjoint from the set of generic leaves for $\mu^j$,  $i\neq j$, and it follows that $\CC(\Lambda)$ is a cone on a simplex.  


\subsection{Transverse decomposition for an unfolding sequence}\label{subsec: transverse decomposition}
Given a graph $G$, a {\em transverse decomposition} of $G$ is a collection $H^0,\ldots, H^k$ of subgraphs of $G$ with
\[ EG = \biguplus_{i=0}^k EH^i, \]
the disjoint union of the edges of $H^0,\ldots,H^k$. The subgraph $H^i$ is {\em degenerate} if $EH^i$ is empty. 

As before assume a reduced unfolding sequence $(G_n)_{n\le0}$ is given with legal lamination $\closure\Lambda$ and probability ergodic currents $\mu^1,\ldots,\mu^k$.  We further assume $G$ is a graph with a given isomorphism to $G_n$ for every $n$. We use this isomorphism and for every edge $e$ of $G$, we denote the corresponding edge in $G_n$ by $e_n$. 

We now claim there is a transverse decomposition of $G$ induced from the currents $\mu^1,\ldots,\mu^k$. As before we assume $(\vec\lambda_n)_n$ is the simplicial length measure on $(G_n)_n$ induced from $G_0$ and for $e\in EG$, we use $\lambda_n(e)$ and $\mu^i_n(e)$ to denote the $\lambda_n$-length of $e_n$ and the $\mu^i_n$-measure of the cylinder ${\rm Cyl}(e_n), i=1,\ldots,k,$ respectively. 
Also $\langle \gamma , \phi_n(e_n) \rangle$ denotes the number of occurrences of $\gamma$ in the path $\phi_n(e_n)$ in $G_0$. 

\begin{thm}\label{transverse decomposition}
	Suppose a reduced unfolding sequence $(G_n)_n$ with legal lamination $\closure\Lambda$ and ergodic probability currents $\mu^1,\ldots,\mu^k$ is given. Also assume for every $n$, $G_n$ is identified with a fixed graph $G$.
	After passing to a subsequence, there is a transverse decomposition $H^0,H^1,\ldots, H^k$ of $G$ so that for every distinct pair $i,j\in\{1,\ldots,k\}$ and $e\in H^i$
	\begin{enumerate}
		\item $\ds \liminf_n \mu^i_{n}(e)\lambda_{n}(e) > 0,$
		\item $\ds \sum_n \mu_n^j(e)\lambda_n(e) < \infty,$ and
		\item \[ \lim_n \frac{\langle \gamma, \phi_n(e_{n})\rangle}{\lambda_{n}(e)} = \mu^i(\gamma),\]
	\end{enumerate}
    Also given $e\in EH^0$ and $i\in \{1,\ldots,k\}$
    \[ \sum_n \mu_n^i(e)\lambda_n(e) <\infty .\]
\end{thm}
  
\begin{rem}
 Conclusion (2) in the statement may seem somewhat artificial; we state it explicitly for convenience.
\end{rem}  
  
\begin{proof}
	We can pass to a subsequence so that for every $e\in EG$, either there exists $i\in\{1,\ldots,k\}$ with $\liminf_n \mu_n^i(e)\lambda_n(e) >0$ or for every $i\in\{1,\ldots,k\}$, $\lim_n\mu_n^i(e)\lambda_n(e)=0$. 
	Define $H^i$ to consist of the edges with 
	\[ \liminf_n\mu_n^i(e)\lambda_n(e)>0\] 
	and $H^0$ to be spanned by edges $e$ with 
	\[ \lim_n\mu_n^i(e)\lambda_n(e)=0\] 
	for every $i\in\{1,\ldots,k\}$. Recall by lemma \ref{pre-image of leaves} that for every leaf $\omega$ of $\closure\Lambda$ and every $n$, there is a unique bi-infinite legal path $\omega_n$ in $G_n$ whose $\phi_n$-image is a relabeling of $\omega$ in the same direction. Moreover the $\phi_n$-image of the edge $\omega_n[0,1]$ is a directed path that contains the edge $\omega[0,1]$ in the same direction.
	
	\begin{lem}\label{large frequency implies containing generic points}
	    If $e\in EG$ is given with $\limsup_n\mu_n^i(e)\lambda_n(e) >0$ for some $i\in\{1,\ldots,k\}$ then there is a subset of positive $\mu^i$-measure of $\closure\Lambda$ with the property that for every leaf $\omega$ in this set $\omega_n[0,1]=e_n$ for infinitely many $n$. 
	\end{lem}
    \begin{proof}
        Given and edge $e\in EG$ and $i$ let
	\[ A_n(e) = \{ \omega\in\closure\Lambda \, : \, \omega_n \in {\rm Cyl}(e_n)\subset \closure\Lambda. \}\]
	It follows that
    	\begin{equation}\label{eq: size of image of a cylinder} 
        \mu^i (A_n(e)) = \mu^i_n(e)\,\lambda_n(e).
    \end{equation}
        for every $n$. So with the assumption of the lemma, infinitely of these sets have $\mu^i$-measure bigger than $\ep>0$ for some fixed $\ep>0$. 
	It is then immediate that 
	\[ \bigcap_m \bigcup_{n\le m} A_n(e) \]
	has $\mu^i$-measure at least $\ep$. This proves the lemma.
    \end{proof}
    
    \begin{lem}\label{containing generic implies generic frequency}
        Given $i\in\{1,\ldots,k\}$, there is a subset of full $\mu^i$-measure in $\closure\Lambda$ with the property that for every $\omega$ in this set
        \[ \lim_n \frac{\langle \gamma, \phi_n(\omega_n[0,1])\rangle}{\lambda_n(\omega_n[0,1])} =\mu^i(\gamma). \]
    \end{lem}
    \begin{proof}
        Given $\omega$ in $\closure\Lambda$ assume $e(\omega,n) = \omega_n[0,1]$ is the edge of $G_n$ that appears at the $[0,1]$ segment of $\omega_n$. Then by the definition of $\omega_n$, $\phi_n(e(\omega,n)) =\omega[a_n,b_n]$ with $a_n\le 0 < b_n$ and $b_n - a_n = \lambda_n(e(\omega,n))$ is the simplicial length of the image in $G_0$. By lemma \ref{length and frequency decay}, $b_n-a_n\to \infty$ as $n\to-\infty$, and by \eqref{eq: generic frequency} for $\mu^i$-almost every $\omega$
        \[ \lim_n\frac{\langle\gamma, \omega[a_n,b_n]\rangle}{b_n-a_n} = \mu^i(\gamma).\]
        This proves the lemma.
    \end{proof}
	
	As a consequence of the above two lemmas, if $e\in EH^i$ is given, i.e.,
	\[ \liminf_n \mu_n^i(e)\lambda_n(e) >0\] 
	and $i\neq j\in\{1,\ldots,k\}$ we must have $\lim_n\mu^j_n(e)\lambda_n(e) =0$; otherwise by lemma \ref{large frequency implies containing generic points} we will end up with a $\mu^i$-generic $\omega^i$ and a $\mu^j$-generic $\omega^j$ so that for infinitely many $n$, $\omega^i_n[0,1] = \omega^j_n[0,1] = e_n$. But then lemma \ref{containing a generic point implies generic} implies that along this subsequence
	\[ \frac{\langle \gamma, \phi_n(e_n) \rangle}{\lambda_n(e_n)}\]
	converges both to $\mu^i(\gamma)$ and $\mu^j(\gamma)$ which is impossible by the assumption that $\mu^i(\gamma)\neq \mu^j(\gamma)$. Once we know that $\lim_n\mu^j_n(e)\lambda_n(e)=0$ for every $j\neq i$, we can obviously pass to further subsequences and assume 
	\[ \sum_n \mu^j_n(e)\lambda_n(e) <\infty.\] 
This shows that edges of $H^i$ satisfy (1), (2), and (3) and moreover $H^i$ and $H^j$ share no edge for $i\neq j$. 
	
	By definition of $H^0$, for every edge $e$ in $H^0$ and every $i\in\{1,\ldots,k\}$, $\lim_n \mu_n^i(e) \lambda_n(e) =0$. After passing to a subsequence we can also assume 
	\[\sum_n\mu_n^i\lambda_n(e) <\infty\]
	for every $i\in\{1,\ldots,k\}$.
\end{proof}

It follows from the above theorem and \eqref{eq: size of image of a cylinder} in the proof that if $e\in EG$ is not in $H^i$ then
    \[ \sum_n \mu^i(A_n(e)) = \sum_n \mu_n^i(e)\lambda_n(e) < \infty.\]
By standard arguments from measure theory, this implies the set of $\omega$ that belong to infinitely many such sets has zero $\mu^i$-measure.
Equivalently the set of $\omega\in\closure\Lambda$ with $\omega_n[0,1] = e_n$ for infinitely many $n$ has zero $\mu^i$-measure. 

\begin{prop}\label{image of generic before pinching}
    With the hypothesis of theorem \ref{transverse decomposition} and after passing to a subsequence for which the conclusion holds, for $\mu^i$-almost every $\omega$ in $\closure\Lambda$, $\omega_n[0,1]\in EH^i_n$ for $n$ sufficiently large (depending on $\omega$). In particular $H^i$ is non-degenerate for every $i\in\{1,\ldots,k\}$.
\end{prop}

Note that $H^i_n$ denotes the subgraph of $G_n$ identified with $H^i$ for every $n\le 0$ and $i\in\{0,1,\ldots,k\}$.


\subsection{Pinching along an unfolding sequence}\label{subsec: pinching}
As before assume $(G_n)_{n\le 0}$ is a reduced unfolding sequence with legal lamination $\closure\Lambda$ and ergodic probability currents $\mu^1,\ldots,\mu^k$. We also assume every $G_n$ is equipped with a metric $\lambda_n$ induced by the simplicial length vector $\vec\lambda_n$. We say the sequence $(G_n,\lambda_n)$ {\em converges in the moduli space of graphs} if there is a fixed graph $G$ with given isomorphisms to $G_n$ for every $n$ and so that for every $e\in EG$
\[ \lim_n \frac{\lambda_n(e)}{\lambda_n(G)} \]
exists, where $\lambda_n(G)$ denotes the total $\lambda_n$-length of $G_n$. We say $G$ is the {\em limit graph} and the {\em length limit} is the above limit for edges of $G$. We define the {\em pinched part} to be the subgraph $E\subset G$ spanned by edges whose length limit is zero.

Given a reduced unfolding sequence $(G_n)_n$, assume we have passed to a subsequence that converges in the moduli space of graphs with limit graph $G$ and pinched part $E$. Moreover and after passing to a further subsequence, we assume the conclusion of theorem \ref{transverse decomposition} holds and we have the transverse decomposition $H^0,H^1,\ldots,H^k$ of $G$ associated to ergodic probability currents $\mu^1,\ldots,\mu^k$. Collapsing the pinched part $E$ of $G$ induces a transverse decomposition $H^0/E, H^1/E, \ldots, H^k/E$ of $G/E$. Note however that it is possible for $H^i/E$ to be degenerate.

Recall that a leaf $\omega$ of $\closure\Lambda$ provides legal labeled bi-infinite paths $\omega_n\in \closure\Omega_\infty(G_n)$ for every $n$, whose $\phi_n$-image is a relabeling of $\omega$ and the $\phi_n$-image of $\omega_n[0,1]$ contains $\omega[0,1]$. Via the isomorphism between $G_n$ and $G$, the bi-infinite labeled path $\omega_n$ identifies with a bi-infinite labeled path in $G$ which is still called $\omega_n$. We say a path $\alpha$ in $G$ is an {\em accumulation of $\omega$} if there are integers $a<b$ and a subsequence of $(\omega_n[a,b])_n$ that converges to $\alpha$. We use $\alpha_n$ to denote the image of $\gamma$ in $G_n$; in particular $\alpha_n = \omega_n[a,b]$ for infinitely many $n$.

\begin{thm}\label{frequency in generic leaves}
    Given $i\in\{1,\ldots,k\}$, for $\mu^i$-almost every $\omega\in\closure\Lambda$, if a path $\alpha$ in $G$ is an accumulation of $\omega$,
    \[ \frac{\lambda_n(\alpha_n \setminus H_n^i)}{\lambda_n(G)} \to 0\]
    as $n\to -\infty$, where $\alpha_n\setminus H_n^i$ is the part of $\alpha_n$ that falls outside of $H_n^i$.
\end{thm}

\begin{proof}
    Given a positive integer $c$ and for an edge $e\in EG$ and leaf $\omega\in\closure\Lambda$ for every $n$, we consider sub-path of $\omega$ given by the interval $\omega[-c\lambda_n(G),c\lambda_n(G)]$ with $\lambda_n(G)$ the total length of $G_n$. Obviously $2c\lambda_n(G)$, the simplicial length of these paths, tend to infinity with $n\to -\infty$. If $\alpha$ is an accumulation of $\omega$, then it is easy to see that we can choose $c$ large enough so that $\phi_n(\alpha_n)$ is a sub-path of $\omega[-c\lambda_n(G),c\lambda_n(G)]$. Motivated by this observation define
    \[ f_n^c(\omega,e) = \frac1{2c\lambda_n(G)} \left| \omega[ -c\lambda_n(G) , c\lambda_n(G)] \cap \phi_n(e_n) \right|, \]
    the ratio of the simplicial length of part of $\omega[-c\lambda_n(G),c\lambda_n(G)]$ which is in $\phi_n(e_n)$. 
    
    Using the comments in the beginning of the proof, it will be enough to show that for $\mu^i$-almost every $\omega$ and an edge of $G$ outside of $H^i$, $f_n^c(\omega,e)\to 0$ as $n\to-\infty$. 
       As before assume $A_n(e)$ denotes the set of $\omega\in\closure\Lambda$ with the property that $\omega_n[0,1]=e_n$ and $\chi_{A_n(e)}$ is the characteristic function of $A_n(e)$ then
    \[ f_n^c(\omega,c) = \frac1{2c\lambda_n(G)} \sum_{j=-c\lambda_n(G)}^{c\lambda_n(G)-1} \chi_{A_n(e)}(S^j\omega).\]
    (Recall that $S$ is the shift map on $\closure\Lambda$.) Note that by \eqref{eq: size of image of a cylinder} and for every $i\in\{1,\ldots,k\}$
    \[  \mu^i(A_n(e)) = \int_{\closure\Lambda}\chi_{A_n(e)}(\omega)\, d\mu^i(\omega) = \mu^i_n(e)\lambda_n(e).\]
    We use this to integrate the function $f_n^c(\omega,e)$:
    \begin{equation}\label{eq: integral of f unfolding}
    \begin{aligned}
        \int_{\closure\Lambda} f_n^c(\omega,e)\, d\mu^i(\omega) & = \frac1{2c\lambda_n(G)} \int_{\closure\Lambda} \left(\sum_{j=-c\lambda_n(G)}^{c\lambda_n(G)-1} \chi_{A_n(e)}(S^j\omega)\right) \, d\mu^i(\omega) \\
        & = \frac1{2c\lambda_n(G)} \sum_{j=-c\lambda_n(G)}^{c\lambda_n(G)-1} \int_{\closure\Lambda} \chi_{A_n(e)}(S^j\omega)\, d\mu^i(\omega) \\
        & = \frac1{2c\lambda_n(G)} \sum_{j=-c\lambda_n(G)}^{c\lambda_n(G)-1} \int_{\closure\Lambda} \chi_{A_n(e)}(\omega)\, d\mu^i(\omega) \\
        & = \mu_n^i(e)\lambda_n(e),
    \end{aligned}
    \end{equation}
    where we have used the shift invariance of $\mu^i$. When $e$ is not in $H^i$, $\sum_n \mu_n^i(e)\lambda_n(e) <\infty$. Therefore by the Monotone Convergence Theorem
    \begin{align*}
    \int_{\closure\Lambda} \left( \sum_n f_n^c(\omega,e) \right) \, d\mu^i(\omega) 
        & = \sum_n \int_{\closure\Lambda} f_n^c(\omega,e)\, d\mu^i(\omega) = \sum _n \mu_n^i(e)\lambda_n(e) < \infty.
    \end{align*}
    This implies for $\mu^i$-almost every $\omega$
    \[ \sum_n f_n^c(\omega,e) \to 0\]
    as $n\to-\infty$ and in particular $f_n^c(\omega,e)\to 0$.
\end{proof}

\begin{cor}\label{projection of generic for unfolding}
    Given $i\in\{1,\ldots,k\}$, for $\mu^i$-almost every leave $\omega$ of $\closure\Lambda$, if $\alpha$ is an accumulation of $\omega$ in $G$, then $\alpha$ projects to $H^i/E$.
\end{cor}

In the proof of theorem \ref{frequency in generic leaves}, we showed that if $e\in EG\setminus EH^i$ then for $\mu^i$-almost every $\omega$, $f_n^c(\omega,e)\to 0$ with $n$. For edges in $H^i$, we easily see the following alternative.

\begin{lem}\label{positive frequency for generic edges}
    Given $e\in EH^i$, $i\in\{1,\ldots,k\}$, and $c>0$, there exists $\ep>0$ and a set of positive $\mu^i$-measure in $\closure\Lambda$ so that for every leaf $\omega$ in this set $f_n^c(\omega,e)>\ep$ for infinitely many $n$.
\end{lem}
\begin{proof}
    By \eqref{eq: integral of f unfolding}, $\int f_n^c(\omega,e)\, d\mu^i = \mu_n^i(e)\lambda_n(e)$ for every edge $e\in EG$. 
    When $e\in EH^i$, $\liminf_n \mu_n^i(e)\lambda_n(e)>0$ and we can find $\ep>0$, so that the $\mu^i$-measure of the set of $\omega$ with $f_n^c(\omega,e)>\ep$ is uniformly bounded from below by a positive constant. Hence there is a set of positive $\mu^i$-measure whose elements belong to infinitely many of these sets and this proves the claim.
\end{proof}


We use this to show one can find a path $\alpha$ in $G$ which is an accumulation of a generic $\omega$ and traverses $e$ as many times as required. This will immediately imply that a vertex of a non-degenerate component of $H^i/E$ has at least valence $2$.

\begin{cor}\label{generic leaves traverse a generic edge many times for unfolding}
    Given an edge $e$ of $H^i, i\in\{1,\ldots,k\}$ and $l>0$, there is a set of positive $\mu^i$-measure in $\closure\Lambda$, so that for every $\omega$ in this set, there is a an accumulation $\alpha$ of $\omega$ in $G$ which traverses $e$ at least $l$ times.
\end{cor}

\begin{cor}\label{valence at least 2 after pinching unfolding sequence}
    Every vertex of a non-degenerate component of $H^i/E, i\in\{1,\ldots,k\}$ has valence at least $2$ (in $H^i/E$); also a component of $H^i$ cannot be contained in a component of $E$ which is a tree.
\end{cor}

\subsection{Recurrence for an unfolding sequence}\label{subsec: recurrence for unfolding}
An unfolding sequence $(G_n)_{n\le0}$ is {\em recurrent} if (equipped with the simplicial lengths $(\lambda_n)_n$) it has a subsequence that converges in the moduli space of graphs and the pinched part for this subsequence is a forest. Equivalently $(G_n)_n$ is recurrent if the projections of $(G_n,\lambda_n)_n$ in $CV_N$ contain a subsequence that stays in a cocompact subset of $CV_N$. 
When $(G_n)_n$ is recurrent, we cannot conclude unique ergodicity but we can bound the dimension of the space of currents supported on the legal lamination.

\begin{thm}\label{recurrent implies bounded number of measures}
	Assume $(G_n)_n$ is a recurrent unfolding sequence and $\closure\Lambda$ is the legal lamination. The dimension of the space of currents supported on $\closure\Lambda$ is at most $N$, i.e., there are at most $N$ mutually singular ergodic currents supported on $\closure\Lambda$.
\end{thm}
\begin{proof}
	Suppose $(G_{n_m})$ is a subsequence that converges in the closure of the moduli space of graphs with limit graph $G$ and the pinched part $E$, which is a forest in $G$. Also assume $H^0,H^1,\ldots,H^k$ is the transverse decomposition of $G$ corresponding to the collection of mutually singular ergodic currents supported on $\closure\Lambda$. 
	
	By corollary \ref{valence at least 2 after pinching unfolding sequence}, no 
$H^i$, $i\in\{1,\ldots,k\}$, is contained in $E$ and after collapsing the pinched part $H^0/E, H^1/E, \ldots, H^k/E$ is a transverse decomposition of $G/E$ with non-degenerate $H^i/E$ for every $i>0$. In addition for $i>0$, every vertex of $H^i/E$ has degree at least $2$ in $H^i/E$ and this immediately implies $k\le N$.
\end{proof}


\section{Folding Sequences}\label{sec: folding sequences}

Recall that a {\em folding sequence} is a sequence 
\begin{equation}\label{eq: folding sequence}
\xymatrix{
G_0 \ar[r]^{f_0} & G_{1} \ar[r]^{f_{1}} \ar[r] & \cdots \ar[r]^{f_{n-1}} & G_n  \ar[r]^{f_n} & \cdots
}
\end{equation}
where every $f_n, n\ge 0$, is a change of marking morphism. We also define the morphism $\phi_n:G_0\to G_n$ to be the composition $f_{n-1}\circ\cdots\circ f_0$.

Similar to unfolding sequences, we assume the sequence $(G_n)_{n\ge0}$ is reduced, i.e., it does not have a stabilized sequence of subgraphs.

In this section, we study the space of length measures on the folding sequence $(G_n)_n$. Not that the folding sequence lifts to a sequence 
\[
\xymatrix{
T_0 \ar[r]^{\cover f_0} & T_{1} \ar[r]^{\cover f_{1}} \ar[r] & \cdots \ar[r]^{\cover f_{n-1}} & T_n  \ar[r]^{\cover f_n} & \cdots
}
\]
where $T_n$ is the tree which is the universal cover of $G_n$. If $(\vec\lambda_n)_n$ is a length measure on $(G_n)_n$ and $\lambda_n$ is the induced metric on $G_n$, then it lifts to a metric for $T_n$ which is still denoted by $\lambda_n$.

Also we will use the frequency current $(\vec\mu_n)_n$ on $(G_n)_n$. Recall that $\vec\mu_0 = \vec {\bf 1}$ and
\[ \vec \mu_{n+1} = M_{f_n}\vec\mu_n \]
where $M_{f_n}$ is the incidence matrix of $f_n$.
For $e\in EG_n$, the component $e$ of $\vec \mu_n$, denoted by $\mu_n(e)$, is the number of times that $\phi_n$-images of edges of $G_0$ traverse $e$.

\subsection{The limit tree}\label{subsec: limit tree}  We define the {\em limit tree} $T$  of the sequence of trees $T_n$. First, define an equivalence relation $\sim_0$ on $T_0$ where $x \sim_0 y$ if there is $p$ so that the images of $x$ and $y$ in $T_p$ coincide.  The quotient space $T_0/\sim_0$ may not be Hausdorff, and so we use $\sim$ to be the equivalence relation whose classes are closures of $\sim_0$-classes.  In other words, $T_0/\sim$ is exactly the maximal Hausdorff quotient of $T_0/\sim_0$.   For every $n$, we have an induced quotient map $\psi_n:T_n\to T$.  Notice that $\FN$ acts by homeomorphisms on $T$.  Using Lemma \ref{length and frequency decay}, one sees that this action has a dense orbit.   Also for every $n$, the $\psi_n$-image of every edge of $T_n$ is an embedded arc in $T$.  One checks that $T$ is uniquely path connected and locally path connected.  Further, it easy to see that $\sim_0$ arises from a pseudo-metric on $T_0$ and that $\sim$ arises from a metric, so $T$ is metrizable.  Hence, by \cite{MO90}, $T$ is an $\mathbb{R}$-tree.  

We say a point of $T_m$ {\em maps to a branch point} if its image in $T_n$ for some $n\ge m$ is a vertex of $T_n$. Clearly the set of such points is countable. 

Given a length measure $(\vec\lambda_n)$ for the sequence we obtain a pseudo-distance function $\lambda$ on $T$ defined so that the $\lambda$-distance between $x$ and $y$ in $T$ is the infimum of the distances of $x_n, y_n\in T_n$ where $\psi_n(x_n)=x$ and $\psi_n(y_n)=y$. This is only a pseudo-distance since it is possible for the distance between distinct points to be zero. However we can identify every such pair and the quotient equipped with the induced metric, denoted by $\lambda$, is an $\BR$-tree with an action of $F_n$ by isometries. This $\BR$-tree represents a point of the boundary of $CV_N$ and we have

\begin{lem}\label{folding sequence converges to the boundary}
	The projection of the marked metric graphs $(G_n,\lambda_n)_n$ to $CV_N$ converge to the $\BR$-tree $(T,\lambda)$ which is obtained from the pseudo-distance $\lambda$ on $T$.
\end{lem}

As a consequence of lemma \ref{length and frequency decay} we have:

\begin{lem}\label{limit tree has dense orbits}
	If $(G_n)_n$ is a reduced folding sequence that admits a length measure, the limit tree with the induced metric has dense orbits.
\end{lem}

Note that passing to a subsequence of $(G_n)_n$ does not change the limit tree and the set of length measures on $(G_n)_n$.

\subsection{Length measures on a tree}\label{subsection: length measures of a tree}
Suppose $T$ is an arbitrary tree. Following Guirardel \cite{Gui00}, we define a {\em length measure} on $T$ to be a collection of finite Borel measures $\lambda_I$ for every compact interval $I$ in $T$ and such that if $J\subset I$, $\lambda_J =(\lambda_I)|_J$. When $T$ is endowed with an action of a group (the free group $F_N$ in our setting), we also assume these measures are invariant under this action. A length measure $\lambda$ is {\em non-atomic} if every $\lambda_I$ is non-atomic. By default a length measure is assumed to be non-atomic and we denote the space of non-atomic invariant length measures on $T$ by $\CD(T)$. 

When $T$ is an $\BR$-tree, the collection of the Lebesgue measures of the intervals in $T$ provide a length measure which is referred to as the {\em Lebesgue measure} of $T$. Also if $T$ is equipped with a pseudo-distance so that after collapsing diameter zero subsets one gets an $\BR$-tree, the pull back of the Lebesgue measure provides a non-atomic measure which we refer to as the Lebesgue measure.
The Lebesgue measure is always non-atomic.

\subsection{The space of length measures}\label{subsec: space of length measures}

When $T$ is the limit tree for a folding sequence $(G_n)_{n\ge0}$, every non-atomic length measure $\lambda$ on $T$ naturally induces a length measure on $T_n$ for every $n$ so that the map $\psi_n$ restricted to an edge of $T_n$ preserves the measure. Since $\lambda$ is invariant under the action of $\pi_1(G_n)$, we also obtain a metric $\lambda_n$ on $G_n$ and a length vector $\vec\lambda_n$ where for every $e\in EG_n$, the component $\lambda_n(e)$ of the vector is equal to the $\lambda_n$-length of $e$. Hence we obtain a sequence of length vectors $(\vec\lambda_n)_n$ for $(G_n)$ and
\[ \vec\lambda_{n+1} = M_{f_n}^T \vec\lambda_n;\]
hence the sequence $(\vec\lambda_n)_n$ is a length measure for the folding sequence.

\begin{lem}\label{length measure map is one to one}
	Given a reduced folding sequence $(G_n)$ and limit tree $T$, if $\lambda^1, \lambda^2$ are distinct non-atomic length measures, for $n$ sufficiently large, the induced length measures $\vec\lambda^1_n, \vec\lambda^2_n$ on $G_n$ are distinct.
\end{lem}
\begin{proof}
	Let $\lambda = \lambda^1+\lambda^2$. Since $\lambda^1\neq\lambda^2$, there is an arc $I$ in $T$ and a subset $S\subset I$ of positive $\lambda$ measure with the property that for every $x\in S$ there exists $\ep>0$ so that if $J\subset I$ is an interval of length $\le\ep$ containing $x$ then $\lambda^1(J)\neq\lambda^2(J)$. Essentially $S$ is the set of points in $I$ where the Radon-Nikodym derivatives $d\lambda^1/d\lambda$ and $d\lambda^2/d\lambda$ are not equal. 

We can assume $I$ is contained in the $\psi_0$-image of an edge $e_0$ of $T_0$. Since the set of points that map to a branch point is countable and $\lambda$ is non-atomoic , we can further assume that $S$ does not have any such point.
Hence for every $n$, we can choose $x\in S$ and an edge $e_n$ of $T_n$ which is in the image of $\phi_n(e_0)$, so that $\psi_n(e_n)$ in $T$ is a sub-interval of $I$ containing $x$ as an interior point. By lemma \ref{length and frequency decay}, the $\lambda_n$-length of $e_n$ tends to zero as $n\to\infty$ and therefore $\lambda^1_n(e_n) \neq \lambda^2_n(e_n)$ for $n$ sufficiently large. This proves the lemma.
\end{proof}

Conversely if the sequence
$(\vec\lambda_n)$ is a length measure for the folding sequence $(G_n)$, we showed that on $T$ one obtains a pseudo-distance and the induced Lebesgue measure is a non-atomic length measure on $T$. Putting these together we have:

\begin{prop}\label{equivalence of length measures}
	Given a reduced folding sequence $(G_n)$ and limit tree $T$, there is a linear isomorphism between the space $\CD((G_n)_n)$ of length measures on $(G_n)$ and the space $\CD(T)$ of invariant non-atomic length measures on $T$.
\end{prop}

This immediately shows that $\CD(T)$ is a positive cone in a vector space whose dimension is bounded by $\liminf_n |EG_n|$ and in particular by $3N-3$. 

Recall that given the length measure $\lambda=(\vec\lambda_n)_n$ and the frequency vectors $(\vec\mu_n)_n$, the area is equal to 
\[ \vec\mu_n^T\lambda_n \]
which is independent of $n$. We define $\CD_0(T)$ to be the subset of $\CD(T)$ consisting of area one length measures, i.e., those $\lambda$ for which 
\[ \vec {\bf 1}^T\lambda_0 = 1.\] 
Note that the above quantity is equal to the total $\lambda$-length of $G_0$ and therefore
\[ \lambda_0(G_0) = \sum_{e\in EG_0}\lambda_0(e) =1.\]
The subset $\CD_0(T)$ is a compact subset which is a section of the cone $\CD(T)$.

\subsection{Ergodic length measures}\label{subsec: ergoidic length measures}
We say a length measure $\lambda\in \CD(T)$ is {\em ergodic} if whenever $\lambda = c_1\lambda^1+c_2\lambda^2$ for $c_1,c_2>0$ and $\lambda^1,\lambda^2\in \CD(T)$, then $\lambda^1$ and $\lambda^2$ are homothetic to $\lambda$. It is a standard consequence of the Radon-Nikodym Theorem that non-homothetic ergodic length measures are mutually singular and therefore linearly independent. This immediately implies the following.

\begin{thm}\label{finite ergodic length measures}
	Given a reduced folding sequence $(G_n)_n$ with limit tree $T$, the number of pairwise non-homothetic ergodic length measures on $T$ is bounded by one less than the dimension of $\CD(T)$ and therefore by $3N-4$.
\end{thm}

This also follows from a result of Guirardel \cite{Gui00} but as we have seen it is easier in this case. Note that every $\lambda\in\CD(T)$ has an {\em ergodic decomposition} $\lambda = \lambda^1+\cdots+\lambda^k$ where each $\lambda^i\in\CD(T)$ is ergodic and $\lambda^1,\ldots,\lambda^k$ are the {\em ergodic components} of $\lambda$. Using Guirardel's terminology when $\lambda$ is ergodic, the limit tree with the length measure induced by $\lambda$ is {\em uniquely ergometric}.

Using the above theorem and for a reduced folding sequence $(G_n)$ with limit tree $T$, assume $\lambda^1, \ldots, \lambda^k$ is a maximal collection of pairwise non-homothetic ergodic length measures on $T$ and define $\lambda = \lambda^1+\cdots+\lambda^k$. 

Suppose $T$ is equipped with the distance given by $\lambda$ and every $G_n$ and $T_n$ is equipped with the induced metric $\lambda_n$. In particular every edge $e$ of $T_0$ is identified with the interval $[0,\lambda_0(e)]$. We say a point $x$ in the interior of $e$ is {\em $\lambda^i$-generic}, $i=1,\ldots,k$, if it does not map to a branch point, and
\begin{equation}\label{eq: generic for length} 
\lim_{\ep\to 0} \frac{\lambda^i_0([x-\ep,x])}\ep = \lim_{\ep\to 0} \frac{\lambda^i_0([x,x+\ep])}\ep = 1.
\end{equation}
Recall that $e$ is identified by the interval $[0,\lambda_0(e)]$ and therefore $[x-\epsilon,x]$ and $[x,x+\epsilon]$ represent subarcs of $\lambda_0$-length $\ep$ of $e$ with endpoint $x$. 

It is a standard consequence of ergodicity that in $T$ (respectively $T_n$ or $G_n$) for $\lambda^i$-almost (resp. $\lambda^i_n$-almost) all points \eqref{eq: generic for length} holds. We sometimes refer to these points as $\lambda^i$-generic. It is also easy to see that the set of $\lambda^i$-generic and $\lambda^j$-generic points are disjoint.

\subsection{Transverse decomposition for a folding sequence}
\label{subsec: transverse decomposition for folding}
Assume (perhaps after passing to a subsequence) that every $G_n$ is identified by a fixed graph $G$. We use this identification and for an edge $e\in EG$, we denote the image in $EG_n$ by $e_n$. Also for $n\ge 0$ and $i=1,\ldots, k$, $\lambda_n(e)$, $\lambda^i_n(e)$, and $\mu_n(e)$ respectively denote the $\lambda_n$-length, $\lambda^i_n$-length, and the corresponding component of the frequency vector $\vec\mu_n$.

\begin{thm}\label{transverse decomposition for folding}
	Suppose a reduced folding sequence $(G_n)$ with pairwise non-homothetic non-atomic ergodic length measures $\lambda^1,\ldots,\lambda^k$ is given and every $G_n$ is isomorphic to a fixed graph $G$. After passing to a subsequence, there is a transverse decomposition $H^0,H^1,\ldots,H^k$ of $G$, such that for a distinct pair $i, j\in\{1,\ldots,k\}$ and $e\in EH^i$
	\begin{enumerate}
		\item $\ds\liminf_n \mu_n(e)\lambda^i_n(e) >0$,
		\item $\ds\sum_n \mu_n(e)\lambda^j_n(e) < \infty$, and
		\item $\ds\lim_n\frac{\lambda^j_n(e)}{\lambda^i_n(e)}=0$.
	\end{enumerate}
	Also given $e\in EH^0$ and $i$ in $\{1,\ldots,k\}$
	\[ \sum_n \mu_n(e)\lambda^i_n(e) < \infty.\] 
\end{thm}

\begin{proof}
    Let $\lambda = \lambda^1+\cdots+\lambda^k$. First we pass to a subsequence, so that for every $e\in EG$, either $\ds\liminf \mu_n(e)\lambda_n(e) >0$ or $\ds\lim \mu_n(e)\lambda_n(e) = 0$. Then since $\lambda = \lambda^1+\cdots+\lambda^k$ and after passing to a subsequence, we can assume for every $e\in EG$, either $\ds \lim \mu_n(e)\lambda_n(e) =0$ or there is $i\in\{1,\ldots,k\}$ so that $\ds \liminf \mu_n(e)\lambda^i_n(e) >0$. 
    
    We define $H^i$ to consist of all edges $e$ with $\ds\liminf \mu_n(e)\lambda^i_n(e) >0$ and define $H^0$ to consist of the edges $e$ with $\ds\lim \mu_n(e)\lambda_n(e)=0$. We proceed to prove this is a transverse decomposition of $G$ and satisfies the conclusion of the theorem.
    
    \begin{lem}\label{decent frequency implies containing generic points}
        If $e\in EG$ has the property that $\limsup_n \mu_n(e)\lambda_n^i(e) >0$ then there is a subset of positive $\lambda^i_0$-measure in $G_0$ so that for every $x$ in this subset $\phi_n(x)\in e_n$ for infinitely many $n$.
    \end{lem}
    \begin{proof}
        Assume $\limsup_n \mu_n(e)\lambda^i_n(e) >\ep>0$ and let 
 \[ B_n(e) = \phi_n^{-1}(e_n)\] 
denote the $\phi_n$-pre-image of $e_n$ in $G_0$. Then
\begin{equation}\label{eq: length of preimage}
	\lambda^i_0(B_n(e)) = \mu_n(e)\lambda^i_n(e)
\end{equation}
for every $n$. This is a consequence of the fact that $\mu_n(e)$ is the number of component of the pre-image of $e_n$ in $G_0$. Hence $\lambda^i_0(B_n(e))>\ep$ for infinitely many $n$. This implies that
\[ \bigcap_m \bigcup_{n\ge m}B_n(e) \]
has $\lambda^i_0$-measure at least $\ep$. Obviously if $x$ is in the above set, then $\phi_n(x)\in e_n$ for infinitely many $n$ and this proves the lemma.
    \end{proof}
    
    \begin{lem}\label{containing a generic point implies generic}
        Suppose $x$ is a $\lambda^i$-generic point of $G_0$ for some $i\in\{1,\ldots,k\}$ and $\phi_n(x)$ is in the edge $e(x,n)$ of $G_n$. Then
        \[ \lim_n\frac{\lambda^i(e(x,n))}{\lambda(e(x,n))} = 1.\]
    \end{lem}
    \begin{proof}
        Since $\phi_n(x)$ is in $e(x,n)$, we can choose a $\phi_n$-pre-image $I_n$ of $e(x,n)$ which is an arc in $G_0$ containing $x$. Because $(G_n)$ is reduced and by lemma \ref{length and frequency decay} $\lambda_0(I_n)=\lambda_n(e(x,n))\to 0$ as $n\to\infty$. Then since $x$ is $\lambda^i$-generic and $\phi_n$ form $I_n$ to $e(x,n)$ is measure preserving with respect to $\lambda$ and $\lambda^i$,
        \[ \lim_n\frac{\lambda^i(e(x,n))}{\lambda(e(x,n))} = \lim_n\frac{\lambda^i(I_n)}{\lambda(I_n)} =1.\]
     \end{proof}
     Combining the above two lemmas we can see that if $\liminf \mu_n(e)\lambda^i_n(e) >0$, then $\lim_n \mu_n(e)\lambda^j_n(e) = 0$ for every $j\neq i$. Otherwise and by using lemma \ref{decent frequency implies containing generic points} we find $\lambda^i$-generic $x$ and $\lambda^j$-generic $y$ so that $\phi_n(x) \in e_n$ and $\phi_n(y)\in e_n$ for infinitely many $n$. This however contradicts lemma \ref{containing a generic point implies generic} and the fact that $\lambda_n(e)\ge \lambda_n^i(e)+\lambda^j(e)$. Using the definition of $H^i$ given above, we know that (1) holds for every $e\in EH^i$. By what we just stated $\lim_n \mu_n(e)\lambda^j_n(e) =0$ for every $j\neq i$; we can pass to a further subsequence and assume in addition that $\sum_n \mu_n(e)\lambda_n^j(e) <\infty$ for every $j\neq i$. Part (3) is an immediate corollary of (1) and (2). 
     
     In the definition of $H^0$ for every edge $e$ of $H^0$, $\lim \mu_n(e)\lambda_n(e) =0$. Again we pass to a further subsequence to assume that for every such $e$, $\sum_n \mu_n(e)\lambda_n(e) <\infty$ and therefore (4) also holds.
\end{proof}

With a simple modification of the above proof, it is possible to start with a length measure on the folding sequence and find a transverse decomposition corresponding to the ergodic components of that length measure.

\begin{thm}\label{transverse decomposition for a length measure}
    Suppose $(G_n)$ is a reduced folding sequence and $\lambda=(\vec\lambda_n)$ is a length measure with ergodic components $\lambda^1,\ldots,\lambda^k$. Also assume every $G_n$ is isomorphic to a fixed graph $G$. After passing to a subsequence, there is a transverse decomposition $H^0,H^1\ldots,H^k$ of $G$, such that for every distinct pair $i,j\in\{1,\ldots,k\}$ and $e\in EH^i$
    	\begin{enumerate}
		\item $\ds\liminf_n \mu_n(e)\lambda^i_n(e) >0$,
		\item $\ds\sum_n \mu_n(e)\lambda^j_n(e) < \infty$, and
		\item $\ds\lim_n\frac{\lambda^j_n(e)}{\lambda^i_n(e)}=0$.
	\end{enumerate}
	Also given $e\in EH^0$
	\[ \sum_n \mu_n(e)\lambda_n(e) <\infty.\] 
\end{thm}

From \eqref{eq: length of preimage} and when the conclusion of theorem \ref{transverse decomposition for folding} holds, we can see that for $i\in\{1,\ldots,k\}$ if $e$ is an edge of $G$ which is not in $H^i$ then
\[ \sum_n \lambda_0(B_n(e)) = \sum_n \mu_n(e)\lambda^i_n(e) < \infty.\]
Using standard measure theory, this implies that the set of points $x$, with $\phi_n(x)\in e_n$ for infinitely many $n$, has $\lambda^i_0$-measure zero. Equivalently:

\begin{prop}\label{generic points fold to corresponding subgraphs}
    For $\lambda^i_0$-almost every $x$ in $G_0$, $\phi_n(x) \in H^i_n$ for $n$ sufficiently large (depending on $x$), where $H^i_n$ is the subgraph of $G_n$ corresponding to $H^i$. In particular $H^i$ is non-degenerate.
\end{prop}

\subsection{Pinching along a folding sequence}\label{subsec: pinching along folding}
	Assume $(G_n)_{n\ge0}$ is a reduced folding sequence and $(\vec\lambda_n)_n$ is a length measure on $(G_n)_n$. Recall that $\vec\lambda_n$ equips $G_n$ with a metric $\lambda_n$. As in \S \ref{subsec: pinching} assume we have passed to a subsequence so that every $G_n$ is isomorphic to a fixed graph $G$ and this subsequence converges in the moduli space of graphs. Also let $E$ denote the pinched part of $G$. After passing to a further subsequence we can assume the conclusion of theorem 
\ref{transverse decomposition for a length measure} holds for a transverse decomposition $H^0,H^1,\ldots,H^k$ associated to the ergodic components $\lambda^1,\ldots,\lambda^k$ of $\lambda$. Collapsing the pinched part of $G$ induces a transverse decomposition $H^0/E,H^1/E,\ldots,H^k/E$ of $G/E$. 

	Assume $x$ is a point of $G_0$ which does not map to a branch point. We say a path $\beta$ in $G$ is an {\em accumulation of germs at $x$} if there is a sequence of paths $I_n$ in $G_0$ containing $x$ and with $\lambda_0(I_n)\to 0$ as $n\to\infty$ so that the path $\phi_{n}(I_{n})$ in $G_n$ is identified with $\beta$ for infinitely many $n$. We use $\beta_n$ to denote the image of $\beta$ in $G_n$. In particular $\beta_n = \phi_n(I_n)$ for infinitely many $n$.

	By proposition \ref{generic points fold to corresponding subgraphs} for $\lambda^i$-almost every $x\in G_0$, if an edge $e$ of $G$ is an accumulation of germs at $x$, $e\in H^i$. A generalization for longer paths holds.

	\begin{thm}\label{accumulations of germs respect decomposition}
		Given $i\in\{1,\ldots,k\}$ for $\lambda^i$-almost every $x\in G_0$, if a path $\beta$ in $G$ is an accumulation of germs at $x$, then
		\[ \frac{\lambda_n( \beta_n \setminus H^i_n)}{\lambda_n(G)} \to 0\]
		as $n\to\infty$, where $\beta_n\setminus H^i_n$ is the subset of $\beta_n$ outside of $H^i$.
	\end{thm}
	
	\begin{proof}
	    Given $c>0$ and $n$, for a point $x$ in the interior of an edge of $G_0$, we define $I_n^c(x)$ to denote the set of all points in the interior of the same edge whose $\lambda$-distance to $x$ is at most $c\,\lambda_n(G)$. Obviously if $t\in I_n^c(x)$ then $x\in I_n^c(t)$. Also given $e\in EG$ and similar to the proof of theorem \ref{transverse decomposition for folding}, we define $B_n(e)$ to be the $\phi_n$-pre-image of $e_n$ in $G_0$. We define

\[ f_n^c(x,e) = \frac1{2c\lambda_n(G)}\lambda(I_n^c(x) \cap B_n(e)) = \frac1{2c\lambda_n(G)}\int_{I_n^c(x)} \chi_{B_n(e)}(t)\, d\lambda(t).\]

After integrating with respect to $\lambda$: 
\begin{equation}\label{eq: integral of f folding}
\begin{aligned}
	\int_{G_0} f_n^c(x,e)\, d\lambda(x) &= \frac1{2c\lambda_n(G)}\int_{G_0}\int_{I_n^c(x)}\chi_{B_n(e)}(t)\, d\lambda(t)\, d\lambda(x) \\
	&= \frac1{2c\lambda_n(G)}\int_{G_0}\chi_{B_n(e)}(t) \int_{I_n^c(t)}d\lambda(x)\, d\lambda(t)\\
	&\le \int_{G_0} \chi_{B_n(e)}(t)\, d\lambda(t)\\
	&= \mu_n(e)\lambda_n(e).
\end{aligned}
\end{equation}
When $e\in H^0$, we know that $\sum_n\mu_n(e)\lambda_n(e) <\infty$. Therefore by the Monotone Convergence Theorem
\[ \int_{G_0}\left(\sum_n f_n^c(x,e)\right)\, d\lambda(x) = \sum_n \int_{G_0}f_n^c(x,e)\, d\lambda(x) \le \sum_n \mu_n(e)\lambda_n(e) <\infty.\]
In particular for $\lambda$-almost every $x$, $\sum_n f_n^c(x,e) <\infty$ and as a result $f_n^c(x,e)\to 0$ as $n\to\infty$.


     On the other hand, assume $i,j\in\{1,\ldots,k\}$ are distinct and $e\in H^j$. By part (3) of theorem 
     \ref{transverse decomposition for a length measure} $\lambda^j_n(e)/\lambda_n(e) \to 1$ as $n\to\infty$. So for a given $\ep>0$ and $n$ sufficiently large we can assume $\lambda^j(e_n)>(1-\ep)\lambda_n(e)$. Using this for every component of the pre-image of $e_n$ in $I_n^c(x)$, we have
     \[ \lambda^j(I_n^c(x)) \ge \lambda^j(I_n^c(x)\cap B_n^i(e)) \ge (1-\ep) \lambda(I_n^c(x)\cap B_n^i(e));\]
     so
     \[ \frac{\lambda^j(I_n^c(x))}{2c\lambda_n(G)} \ge (1-\ep) f_n^c(x,e).\]
     But since $2c\lambda_n(G) = \lambda(I_n^c(x))\to 0$ as $n\to\infty$ and by \eqref{eq: generic for length}, for $\lambda^i$-generic $x$ the fraction on the left tends to zero as $n\to\infty$. Hence $f_n^c(x,e)\to 0$ as $n\to\infty$ for $\lambda^i$-almost every $x\in G_0$.

We have shown that for $\lambda^i$-almost every $x\in G_0$ and $c>0$ fixed,
\[ \frac1{\lambda(I_n^c(x))} \lambda(I_n^c(x)\cap \phi_n^{-1}(G_n \setminus H^i_n)) \to 0\]
as $n\to\infty$. Equivalently for $\lambda^i$-almost every $x\in G_0$, $1/\lambda_n(G)$ times the $\lambda_n$-length of the intersection of $\phi_n(I_n^c(x))$ and the complement of $H_n^i$ tends to zero as $n\to\infty$:
\[ \lim_n \frac{\lambda_n( \phi_n(I_n^c(x)\setminus H^i_n))}{\lambda_n(G)} =0.\]

Given $\beta$ in $G$ which is an accumulation of germs at $x$, we can choose $c>0$, so that the path $\beta_n$ in $G_n$ is contained in $\phi_n(I_n^c(x))$ for every $n$. Then it follows that for $\lambda^i$-almost every such $x$
\[ \lim \frac{ \lambda(\gamma_n\setminus H^i_n)}{\lambda_n(G)} = 0.\]
	\end{proof}

\begin{cor}\label{accumulation of germs after pinching}
    Given $i=1,\ldots,k$, for $\lambda^i$-almost every $x\in G_0$ if a path $\beta$ in $G$ is an accumulation of germs at $x$, then $\beta$ projects to $H^i/E$.
\end{cor}

\begin{lem}\label{frequency in an accumulation}
	Given $e$ in $H^i, i\in\{1,\ldots,k\}$ and $c>0$, there is a set of positive $\lambda^i$-measure in $G_0$, so that for every $x$, $f_n^c(x,e)>\ep$ for infinitely many $n$ and some $\epsilon>0$. 
\end{lem}
\begin{proof}
	By \eqref{eq: integral of f folding} in the previous proof, we know that 
\[ \int_{G_0}f_n^c(x,e)\, d\lambda^i(x) = \mu_n(e)\lambda_n^i(e).\] 
When $e$ is in $H^i$, $\liminf \mu_n(e)\lambda_n^i(e)>0$. From this we can find $\ep>0$ and a set of positive $\lambda^i$-measure, so that $f_n^c(x,e) >\ep$ for infinitely many $n$.
\end{proof}

\begin{cor}\label{recurrence after pinching}
	Given $e$ in $H^i, i\in\{1,\ldots,k\}$ and $l>0$, there is a set of positive $\lambda^i$-measure in $G_0$, so that for every $x$ in this set there is a path $\beta$ in $G$ which is an accumulation of germs at $x$ and $\beta$ traverses $e$ at least $l$ times.
\end{cor}

Corollaries \ref{accumulation of germs after pinching} and \ref{recurrence after pinching} imply in particular that:

\begin{cor}\label{valence two after pinching and folding}
	Given $i\in\{1,\ldots,k\}$, in every non-degenerate component of $H^i/E$, every vertex has valence at least $2$ (in $H^i/E$).
\end{cor}


\subsection{Recurrence for a folding sequence}\label{subsec: recurrence for folding}
As in the case of an unfolding sequence, a folding sequence $(G_n)_{n\ge0}$ with a length measure $\lambda = (\vec\lambda_n)_n$ is {\em recurrent} if it has a subsequence that converges in the moduli space of graphs and the pinched part is a forest. 

\begin{thm}\label{recurrence for folding}
    Assume $(G_n)_n$ is a reduced folding sequence and equipped with a length measure $\lambda=(\vec\lambda_n)_n$, it is recurrent. Then the ergodic decomposition of $\lambda$ consists of at most $N$ components.
\end{thm}

\section{Folding/unfolding sequences and progress in $\CF\CF(N)$}\label{sec: progress in FF}

We say a folding/unfolding sequence $(G_n)_{a\le\le b}$ {\em moves linearly} or is {\em linear} if there exists $M>0$, so that for every $n$ every entry of the incidence matrix of $f_n$ is bounded by $M$. From this one can easily conclude that

\begin{lem}\label{linear implies linear speed}
	Given a linear folding/unfolding sequence, the distances in $CV_N$ grow at most linearly, i.e.,
	\[ d_{CV_N}(G_m, G_n) \le O(n-m).\]
\end{lem}

\begin{rem}
In fact, when $(G_n)_n$ is cocompact, we can show the distances grow at least linearly. 
\end{rem}

Suppose $(G_n)_{n\le 0}$ is an unfolding sequence with legal lamination $\closure\Lambda$ and ergodic probability currents $\mu^1,\ldots,\mu^k$. Recall that $\mu^i$-almost every leaf $\omega$ of $\closure\Lambda$ are generic, i.e.,
\[ \lim_n \frac{\langle\gamma,\omega[0,n]\rangle}n = \lim_n\frac{\langle\gamma,\omega[-n,0]\rangle}n = \mu^i(\gamma)\]
for a fixed choice of $\gamma\in\Omega(G_0)$ (or for every $\gamma$ if one insists). The following lemma shows for $a\le 0<b$ and $b-a$ large, a large (relative to $b-a$) subsegment of $\omega[a,b]$ also has the property that the frequency of occurrence of $\gamma$ is close to $\mu^i(\gamma)$. 

	\begin{lem}\label{frequency in a subpath}
		Given $\theta, \ep>0$ and $\mu^i$-generic leaf $\omega$ of $\closure\Lambda$ there exists $n_0>0$ so that if $a\le 0\le b$ are given with $b-a\ge n_0$ and $[c,d]\subset[a,b]$ is a sub-interval with $c-d > \theta (b-a)$ 
		\[ \left| \frac{ \langle \gamma , \omega[c,d] \rangle}{d-c} - \mu^i(\gamma) \right| < \ep.\]
	\end{lem} \qed

We have shown in theorem \ref{transverse decomposition} that there is a subsequence $(G_{n_m})_m$ of $(G_n)_n$ with $G_{n_m}$ isomorphic to a fixed graph $G$ and there is a transverse decomposition $H^0,H^1,\ldots,H^k$ of $G$ associated to the ergodic currents $\mu^1,\ldots,\mu^k$ so that the conclusion of the theorem holds. 

As usual we use the simplicial length $(\vec\lambda_n)_n$ for every $G_n$. To simplify notation, we also use $\vec{P\lambda_n}$ to denote a multiple of $\vec\lambda_n$ whose component associated to $e\in EG_n$ is $P\lambda_n(e)=\lambda_n(e)/\lambda_n(G_n)$ and $\lambda_n(G_n)$ is the total length of $G_n$.

Given a positive integer $l>0$, we assume we have passed to a further subsequence of $(G_{n_m})_m$ so that for every $m, p=0,\ldots,l$, and $e\in EG_{n_m+p}$ either $P\lambda_n(e)$ is bounded from below by a positive constant or tends to zero with $m$. More precisely, there is a subgraph $E_{n_m+p}\subset G_{n_m+p}$ (another sequence of ``pinched'' subgraphs) so that given $e_{m,p}\in EG_{n_m+p}$, $P\lambda_{n_m+p}(e_{m,p})\to 0$ as $m\to\infty$, also given $e\in EG_{n_m+p}$ which is not in $E_{n_m+p}$, $P\lambda_{n_m+p}(e)>\ep$ for a uniform constant $\ep>0$. Given $m$ and $p=0,\ldots,l$ we can use the composition 
\[ f_{n_m+p-1} \circ \cdots \circ f_{n_m+1} \circ f_{n_m} : G_{n_m}\to G_{n_m+p}\]
and compose it with the collapse map $G_{n_m+p}/E_{n_m+p}$. Assuming that the sequence $(G_n)_n$ is linear we claim:

\begin{lem}\label{disjoint images after collapse for unfolding}
	For $m$ sufficiently large depending on $l$, and a distinct pair $i,j\in\{1,\ldots,k\}$, the images of $H^i_{n_m}$ and $H^j_{n_m}$ in $G_{n_m+p}/E_{n_m+p}$ do not share any edge.
\end{lem}
\begin{proof}
	Assume $e_1\in EH^i_{n_m}$ and $e_2\in EH^j_{n_m}$ are given and $e'$ is an edge of $G_{n_m+p}/E_{n_m+p}$ which is in the images of both $e_1$ and $e_2$. This implies there are sub-arcs $e_1'\subset e_1$ and $e_2'\subset e_2$ and the restrictions of the map $G_{n_m}\to G_{n_m+p}/E_{n_m+p}$ to those are homeomorphisms onto $e'$. Since $(G_n)_n$ is linear, it follows that $ \lambda_{n_m}(G_{n_m}) < (MN)^p \lambda_{n_m+p}(G_{n_m+p})$ with $M$ the upper bound for the entries of the incidence matrices. Since $e'$ is not in $EG_{n_m+p}$, its length is at least $\ep \lambda_{n_m+p}(G_{n_m+p})$. Therefore lengths of $e_1'$ and $e_2'$ are at least $\ep \lambda_{n_m}(G_{n_m})/(MN)^p$. This implies that $\lambda_{n_m}(e_1')>\theta \lambda_{n_m}(e_1)$ and $\lambda_{n_m}(e_2')>\theta\lambda_{n_m}(e_2)$ with
	\[ \theta = \ep/(MN)^l. \]
	Then we can use lemma \ref{frequency in a subpath} to conclude that 
	\[ \frac{\langle \gamma, \phi_{n_m}(e_1') \rangle}{\lambda_{n_m}(e_1')} \to \mu^i(\gamma) \quad {\rm and} \quad \frac{\langle \gamma, \phi_{n_m}(e_2') \rangle}{\lambda_{n_m}(e_2')} \to \mu^j(\gamma)\]
	as $m\to\infty$. But $\phi_{n_m}(e_1') = \phi_{n_m}(e_2')$ since they both factor through $e'$ and this contradicts the fact that $\mu^i(\gamma)\neq \mu^j(\gamma)$.
\end{proof}

For $m$ sufficiently large and $p=0,\ldots,l$, choose an edge $e_{m,p}$ of $G_{n_m+p}$ which is not in $E_{n_m+p}$. By the above lemma, $e_{m,p}$ is in the image of at most one $H^i_{n_m}$, for $i=1,\ldots,k$. So if $k>1$, we can choose $j\in\{1,\ldots,k\}$ so that the image of $H^j_{n_m}$ in $G_{n_m+p}$ does not contain $e_{m,p}$. By corollary \ref{valence at least 2 after pinching unfolding sequence} $H^j_{n_m}/E_{n_m}$ contains a loop, we assume $\delta$ is a primitive loop in $G_{n_m}$ which projects to this loop. Then we claim that for every $p\in 0,\ldots, l$ (and $m$ sufficiently large), the image of $\delta$ in $G_{n_m+p}$ does not fill (misses at least one edge). The idea is that every edge of $\delta$ is either in $H^j_{n_m}$ or is in $E_{n_m}$. We had chosen $e_{m,p}$ so that it is not in the image of $H^j_{n_m}$. Also for $m$ sufficiently large, $e_{m,p}$ is not in the image of $E_{n_m+p}$ because (by the argument in the proof of the previous lemma), length of $e_{m,p}$ is at least $\ep\lambda_{n_m}(G_{n_m})/(MN)^p$. But for edges of $E_{n_m+p}$ the ratio of their length over $\lambda_{n_m}(G_{n_m})$ goes to zero as $m\to\infty$. So $e_{m,p}$ cannot be in the image of $E_{n_m+p}$. 

We have showed that for every $l$, if $m$ is chosen to be large enough, then there exists a primitive loop $\delta$ in $G_{n_m}$ whose image does not fill $G_{n_m+p}$ for every $p\in\{0,\ldots,l\}$. This implies that $d_{\CF\CF}(\delta,G_{n_m+p})\le 3$ where the distance is the $\CF\CF$-distance between the cyclic free factor generated by $\delta$ and a projection of $G_{n_m+p}$ to the free factor complex (obtained by taking the free factor represented by removing an edge of $G_{n_m+p}$). Consequently $d_{\CF\CF}(G_{n_m},G_{n_m+p}) \le 6.$

\begin{prop}
    Given $l$, for $m$ sufficiently large and every $p\in\{0,\ldots,l\}$
    \[ d_{\CF\CF}(G_{n_m},G_{n_m+p}) \le 6.\]
\end{prop}

We have proved:

\begin{thm}\label{slow progress in FF for unfoldings}
    Given a linear unfolding sequence $(G_n)_{n\le 0}$. If $(G_n)_n$ is not dually uniquely ergodic, i.e., the legal lamination of $(G_n)_n$ admits more than one non-homothetic ergodic currents, then for every $l$, there exists $n$ large so that
    \[ d_{\CF\CF}(G_n, G_{n+p})\le 6\]
    for every $p\in\{0,\ldots,l\}$.
\end{thm}

Now assume $(G_n)_{n\ge0}$ is a reduced folding sequence with limit tree $T$ and length measure $\lambda=(\vec\lambda_n)_n$. We also assume every graph $G_n$ is equipped with the metric induced by $\lambda$ and $\lambda^1,\ldots,\lambda^k$ are the ergodic components of $\lambda$. Similar to lemma \ref{frequency in a subpath} we show that for a sub-interval of significant length in a small interval about a $\lambda^i$-generic point $x\in G_0$, most of the length of the sub-interval is induced from $\lambda^i$. More precisely:

\begin{lem}\label{length in a subinterval}
  Given $\theta, \ep>0$ for $\lambda^i$-almost every $x\in G_0$, there exists $\delta>0$, so that if $I$ is an interval of $\lambda$-length at most $\delta$ in $G_0$ containing $x$ and $J\subset I$ is a sub-interval of $\lambda$-length at least $\theta\lambda(I)$, then
  \[ \frac{\lambda^i(J)}{\lambda(J)} > 1-\ep.\]
\end{lem}\qed

We use theorem \ref{transverse decomposition for a length measure} to find a subsequence $(G_{n_m})_m$ of $(G_n)_n$ with each $G_{n_m}$ isomorphic to a fixed graph $G$ and a transverse decomposition $H^0,H^1,\ldots, H^k$ of $G$ which satisfies the conclusion of theorem \ref{transverse decomposition for a length measure}. We use the normalized metric $P\lambda_n$ for the graph $G_n$ which is obtained by dividing the metric $\lambda_n$ by the total length of $G_n$, $\lambda_n(G_n)$.

Given a positive integer $l$, we assume a further subsequence has been chosen so that if for $m$ and $p\in\{0,\ldots,l\}$ an edge $e_{m,p}$ of $G_{n_m+p}$ is chosen then either $P\lambda_{n_m+p}(e_{m,p}) >\ep$ for a constant $\ep>0$ independent of $m$ or $P\lambda_{n_m+p}(e_{m,p})\to 0$ with $m$. We define $E_{n_m+p}\subset G_{n_m+p}$ to contain edges of $G_{n_m+p}$ whose $P\lambda_{n_m+p}$-length is at most $\ep$. (This is the pinched part of the graph.) Given $m$ and $p\in\{0,\ldots,k\}$ we consider the composition of the morphism
\[ f_{n_m+p-1} \circ \cdots \circ f_{n_m+1} \circ f_{n_m} : G_{n_m}\to G_{n_m+p}\]
with the collapse map $G_{n_m+p}\to G_{n_m+p}/E_{n_m+p}$. Analogous to lemma \ref{disjoint images after collapse for unfolding} and when $(G_n)_n$ is linear we claim

\begin{lem}\label{disjoint images after collapse for folding}
	For $m$ sufficiently large (depending on $l$), $0\le p\le l$, and a distinct pair $i,j\in\{1,\ldots,k\}$, the images of $H^i_{n_m}$ and $H^j_{n_m}$ in $G_{n_m+p}/E_{n_m+p}$ do not share any edges.
\end{lem}
\begin{proof}
	The proof follows the same outline as the proof of lemma \ref{disjoint images after collapse for unfolding}. Assume images of $e_1\in EH^i_{n_m}$ and $e_2\in EH^i_{n_m}$ in $G_{n_m+p}$ both traverse an edge $e'$ outside of $E_{n_m+p}$. We choose subintervals $e_1'\subset e_1$ and $e_2'\subset e_2$ which are mapped isometrically to $e'$.
Using linearity of $(G_n)_n$ there exists $\theta$ independent of $m$, so that
\[ \lambda_{n_m}(e_1')>\theta \lambda_{n_m}(e_1) \quad {\rm and} \quad \lambda_{n_m}(e_2') > \theta \lambda_{n_m}(e_2).\]
Then we use lemma \ref{length in a subinterval} to conclude that
\[ \left| \frac{\lambda^i_{n_m}(e_1')}{\lambda_{n_m}(e_1')} - 1\right| \quad {\rm and} \quad \left| \frac{\lambda^i_{n_m}(e_2')}{\lambda_{n_m}(e_2')} - 1\right| \]
will be arbitrarily small for $m$ large. But $e_1'$ and $e_2'$ both map to $e'$ and therefore their $\lambda, \lambda^i,$ and $\lambda^j$-lengths are equal. This contradicts the fact that $\lambda^i+\lambda^j \le \lambda$ for $m$ sufficiently large.
\end{proof}

With this lemma we proceed as in the case of linear unfolding sequences and prove that for $m$ sufficiently large, there is a primitive loop $\delta$ whose image in $G_{n_m+p}$ does not fill. This implies that

\begin{prop}
	Given a positive integer $p$, if $m$ is chosen large enough then for every $p\in\{0,\ldots,l\}$
\[ d_{\CF\CF}(G_{n_m},G_{n_m+p})\le 6.\]
\end{prop}

And we have proved:

\begin{thm}\label{slow progress in FF for foldings}
    Given a linear folding sequence $(G_n)_{n\ge 0}$. If $(G_n)_n$ is not uniquely ergometric, i.e., the limit tree admits more than one non-homothetic non-atomoic length measures, then for every $l$, there exists $n$ large so that
    \[ d_{\CF\CF}(G_n, G_{n+p})\le 6\]
    for every $p\in\{0,\ldots,l\}$.
\end{thm}

\section{Applications}\label{applications}

We deduce two consequences of our main results.  First, we will need to collect some background material.  For the remainder of the paper $N \geq 3$.

\subsection{Outer space}

We give a very brief review of Outer space; for more details, see \cite{FM11, Bes11, BF14}.  Use $\overline{cv}_N$ to denote the set of \emph{very small} actions of $\FN$ on $\mathbb{R}$-trees; elements of $\overline{cv}_N$ are called trees.  Each $T \in \overline{cv}_N$ has an associated translation length function $l_T:\FN \to \mathbb{R}:g \mapsto \inf_{x \in T} d_T(x,gx)$.  The function $\overline{cv}_N \to \mathbb{R}^\FN$ is injective \cite{CM87}, and $\overline{cv}_N$ is given the subspace topology, where $\mathbb{R}^\FN$ has the product topology.  The subspace $cv_N$ consists of trees $T$ with a free, discrete action; the quotient $T/\FN$ is then a metric graph that comes with an identification $\pi_1(T/\FN) \cong F_N$ that is well-defined up to inner automorphisms of $\FN$.  Associated to $T \in cv_N$ is cone on an open simplex got by equivariantly varying the lengths of edges of $T$ and constraining that each edge retain positive length.

Let $T, T' \in cv_N$.  Notice that a change of marking $f:T/F_N \to T'/F_N$ has lifts $\cover{f}:T \to T'$ that are $F_N$-equivariant and which restrict to linear functions on edges; conversely given any such function $g:T \to T'$, one obtains a change of marking $\overline{g}:T/F_N \to T/F_N$.  So, we will use the term \emph{change of marking} to refer either to a change of marking as defined before or as a lift  of a change of marking.  Notice that the Lipschitz constant $Lip(f)$ of $f$ is equal to the maximum slope of $f$ restricted to an edge of $T$.  Define:
\[
d(T,T')=\inf \log(\text{Lip}(f))
\]
\noindent Where $f$ ranges over all change of marking maps $T \to T'$.  The function $d$ is called the \emph{Lipschitz distance} on $cv_N$; $d$ is positive definite and satisfies the triangle inequality, but is not symmetric.  A change of marking $f:T\to T'$ is called \emph{optimal} if $d(T,T')=\log Lip(f)$; optimal maps always exist.  

Assume that $f:T \to T'$ is optimal.  The topological tree $T$ with the $f$-pullback metric from $T'$ is $T_0$, and there is a path in a simplex of $cv_N$ from $T$ to $T_0$.  The induced function $f_0:T_0 \to T'$ is a \emph{morphism}--$f_0$ has slope one of every edge of $T_0$.  The tree $T_0$ can be \emph{folded} according to the map $f_0$: if two edges $e,e'$ of $T_0$ with the same initial vertex are identified by $f_0$, then one can form $T_\epsilon^{e,e'}$ by identifying the the $\epsilon$ initial segments of these edges, and there is an induced morphism $f_\epsilon^{e,e'}:T_\epsilon^{e,e'} \to T'$.  Folding all edges that are identified by $f_0$ in this way gives a \emph{greedy folding path} $T_t$, which we parameterize by $vol(T_t/F_N)$.  For $0\leq s<t\leq \beta$, one has a morphism $f_{s,t}:T_s \to T_t$ so that the equalities hold: $f_{0,t}=f_{s,t} \circ f_{0,s}$ and $f_{s, \beta}=f_{t,\beta}\circ f_{s,t}$.  

A \emph{liberal folding path} is defined similarly to a greedy folding path, except that one does not require to fold all illegal turn simultaneously; see the Appendix of \cite{BF14}.  Our reduced folding/unfolding paths are exactly the restriction of a liberal folding path to a discrete subspace of the parameter space.  

Indeed, starting with a folding/unfolding path $(G_n)_n$, we may insert Stallings factorizations if necessary and then parameterize the new folding/unfolding sequence as a continuous process.  

\begin{lem}\label{2 gates}
Let $(G_n)_n$ be a reduced folding/unfolding path.   For every $n$ and every edge $e$ of $G_n$ with terminal vertex $v$, there is an edge $e'$ with initial vertex $v$ such that the path $ee'$ is legal.
 \end{lem}
 
 \begin{proof}
  Assume that an oriented edge $e$ of $G_n$ with terminal vertex $v$ is such that there is no edge $e'$ of $G_n$ with initial vertex $v$ such that $ee'$ is legal; call the oriented edge $e$ a $v$-dead-end.  Notice that for $k<n$, the only edges that can fold over $e$ via $f_{k,n}$ are also dead ends.  The collection of all dead ends in $G_k$ that fold over $e$ via $f_{k,n}$ is than an invariant subgraph, which much be proper, since no leaf of $\Lambda$ can cross any such edge.
\end{proof}

Hence, legal paths in $G_n$ can be extended; in particular, there are legal loops.  By the definition of the metrics $\vec \lambda_n$ and by Lemma \ref{2 gates}, we have that $\phi_n$ is an optimal map.  Conversely, any liberal folding path can be discretized to give a folding/unfolding path.  Hence, we use the term folding path to refer to liberal folding path or a folding/unfolding path.  

Regard $CV_N \subseteq cv_N$ as graphs with volume one; $CV_N$ is called \emph{Outer space} \cite{CV86}.  We also consider the images of folding paths in $CV_N$, which are also called folding paths.  These paths, which are read left to right, are geodesics and always are parameterized with respect to arc length.  

The $\epsilon$-\emph{thick part}, denoted $CV_N^\epsilon$, of $CV_N$ consists of those $T \in CV_N$ such that the injectivity radius of $T/F_N$ is at least $\epsilon$; $cv_N^\epsilon$ consists of those trees that are homothetic to tree in $CV_N^\epsilon$.  Let us observe that if $T_t$ is a folding path in $CV_N$ and if $T_t \in CN_N^\epsilon$ for all $t$, then by choosing $t_{n+1}-t_n$ to be bounded, we get a linear folding/unfolding path $G_n=T_{t_n}/\FN$.  

According to Corollary 2.10 of \cite{AK11}, there is a constant $C=C(\epsilon)$ such that $d$ is $C$-bi-Lipschitz equivalent to $\check{d}(T,U)=d(U,T)$ on $CV_N^\epsilon$.  Hence, we also have that $d$ is bi-Lipschitz equivalent to $d_s=\max\{d, \check{d}\}$; $d_s$ is a metric, but it is not geodesic.  The action of $Out(\FN)$ on $CV_N^\epsilon$ is co-compact,  so both $(CV_N,d)$ and $(CV_N,d_s)$ are quasi-isometric to $Out(F_N)$.  

\begin{lem}\label{thick distance}
 There is $B(\epsilon)$ such that if $T \in CV_N$ and $U \in CV_N^\epsilon$, then $d(U,T) \leq B(\epsilon)+ C(\epsilon)d(T,U)$.    
\end{lem}

\begin{proof}
 Let $\sigma_T, \sigma_U \subseteq CV_N$ be the open simplices containing $T, U$, respectively.  Since $U \in CV_N^\epsilon$, $d(U,\cdot):\sigma_U \to \mathbb{R}$ is bounded by some constant $A(\epsilon)$.  Let $T' \in CV_N^\epsilon \cap \sigma_T$, then one has $d(U,T) \leq d(U,T') +d(T',T) \leq C(\epsilon)d(T',U)+A(\epsilon) \leq C(\epsilon)(d(T',T)+d(T,U))+A(\epsilon)\leq C(\epsilon)(A(\epsilon)+d(T,U))+A(\epsilon)$.  
 \end{proof}
 
If $H \leq F_N$ is finitely generated and non-trivial, then $T_H$ denotes the minimal $H$-invariant subtree of $T$; so $T_H/H$ is the core of the $H$-cover of $T/F_N$, and we have an immersion $T_H/H \to T/F_N$ representing the conjugacy class of $H$.  The graph $T_H/H$ inherits a metric from $T$, and if a legal structure is specified on $T$, then there is an obvious induced legal structure on $T_H/H$.  If $g \in F_N$ is non-trivial, then $T_g=T_{\langle g \rangle}$ is well-defined, and the immersion $T_g/g$ represents the conjugacy class of $g$.  The volume of $T_g/g$ is the translation length $l_T(g)$ of the loxodromic isometry $g$.  

A result of Tad White, see \cite{FM11, AK11}, gives that 
\[
d(T,U)=\log(\sup_{1 \neq g \in F_N} l_U(g)/l_T(g))
\]
\noindent Moreover, the supremum is realized by a conjugacy class whose simplicial length in $T$ is uniformly bounded.  Specifically, use $c(T)$ to denote the set of conjugacy classes of non-trivial elements $g \in F_N$ that satisfy: 

\begin{itemize}
\item $T_g/g \to T/F_N$ has image a subgraph of rank at most two, 
\item a point in the interior of an edge of $T/F_N$ has at most two pre-images in $T_g/g$, and 
\item the set of edges of $T/F_N$ whose interiors contain a point with two pre-images in $T_g/g$ either is empty or else is topologically a segment, whose interior separates the image of $T_g/g$ into two components, both of which are circles.  
\end{itemize}

The elements of $c(T)$ are called \emph{candidates}.  Tad White proved that
\[
d(T,U)=\log(\max_{g \in c(T)} l_U(g)/l_T(g))
\]

Notice that since $N>2$, the image in $T/F_N$ of the immersion representing any $g \in c(T)$ is contained in a proper subgraph of $T/F_N$.  

\subsection{The complex of free factors}
 
 A \emph{factor} is a conjugacy class of non-trivial proper free factors of $F_N$.  The \emph{complex of free factors} is the simplicial complex whose $(k-1)$-simplices are collections of factors $[F^1],\ldots, [F^k]$, which have representatives satisfying, after possibly reordering, $F^1<\ldots<F^k$.  Let $\mathcal{FF}$ denote the complex of free factors, and equip $\mathcal{FF}$ with the simplicial metric, denoted $d_{\mathcal{FF}}$.  Bestvina and Feighn established the following:
 
 \begin{prop}\cite{BF11}
  The metric space $(\mathcal{FF},d_{\mathcal{FF}})$ is hyperbolic.
 \end{prop} 
 
An essential tool in their approach is a coarsely defined projection map from $cv_N$ to $\mathcal{FF}$.  This projection, denoted $\pi$, is defined as follows: $\pi:cv_N \to \mathcal{FF}:T \mapsto \mathcal{F}(T)$, where $\mathcal{F}(T)$ is the set of factors that are represented by subgraphs of $T/F_N$.  

\begin{prop}\cite{BF11}\label{projection properties}
 There is a number $K$ such that for any $T,U \in cv_N$, $\text{diam}(\pi(T)\cup \pi(U)) \leq Kd(T,U)+K$.
\end{prop}

The conclusion of Proposition \ref{projection properties} is that $\pi$ is coarsely well-defined and coarsely Lipschitz in the sense that $\pi(T)$ has bounded diameter and that distances between the sets of images behave like distances under a Lipschitz map.  


Now suppose that $F'$ is a factor and that $T_t$ is the image of a greedy folding path to $CV_N$ and that this path has been parameterized by arc length.  An immersed segment in $T_H/H$ is called \emph{illegal} if it does not contain a legal subsegment of length at least 3.   Following \cite{BF14}, define $I=(18m(3N-3)+6)(2N-1)$ and put:

\[
\text{left}_{T_t}(F)=\inf \{t|T_H/H \text{contains an immersed legal segment of length} > 2 \}
\]
\[
\text{right}_{T_t}(F)=\sup\{t|T_H/H \text{contains an immersed illegal segment of length} \leq I\}
\]

When the folding path $T_t$ is understood, we use $\text{left}(\cdot)$ to mean $\text{left}_{T_t}(\cdot)$.  When $U \in CV_N$, we define $\text{left}(U)=\min \{\text{left}(F')|F' \in \mathcal{F}(U)\}$ and $\text{right}(U)=\max \{\text{right}(F'|F' \in \mathcal{F}(U)\}$.  The \emph{left projection} of $U$ to $T_t$ is $\text{Left}(U)=T_{\text{left}(U)}$, and the \emph{right projection} of $U$ to $T_t$ is $\text{Right}(U)=T_{\text{right}(U)}$.  We will need the following simple observation:

\begin{lem}\label{nearest point}
 Let $U \in CV_N$ and let $T_t \in CV_N$, $t\in [a,b]$ be the image in $CV_N$ of a greedy folding path, and suppose that $T_c$ minimizes $d(U,\cdot):T_t \to \mathbb{R}$.  If $d(U,T_c)$ is large, then $c \in [\text{left}(U),\text{right}(U)]$.  
\end{lem}

The meaning of the word \emph{Large} is clear from the following proof; it seems uninformative to explicitly specify the constants involved. 

\begin{proof}
 By definition of $\text{left}(U)$, for $t<\text{left}(U)$, no element $g \in c(U)$ has image in $T_t$ containing a legal segment of length at least three, hence the number of illegal turns in $(T_t)_g/g$ is at least $l_{T_t}(g)/2$.  By the Derivative Formula \cite[Lemma 4.4]{BF11}, we have that $l_{T_t}(g)>l_{T_{\text{left}(U)}}(g)$.  
 
 By the definition of $\text{right}(U)$, for $t>\text{right}(U)$, every element $g \in c(U)$ contains a bounded number of illegal segments, so there are at least $l_{T_t}(g)/3I$ disjoint legal segments of length at least 3.  Again, by the Derivative Formula \cite[Lemma 4.4]{BF11}, we have that $l_{T_t}(g)>l_{T_{\text{right}(U)}}(g)$.
\end{proof}

For the remainder of the paper \emph{uniform quasi-geodesic} will mean a quasi-isometric embedding of an interval of $\mathbb{R}$ with a fixed constant.  We keep in mind that set-valued maps are acceptable in the quasi-isometric category as long as the image of every point has uniformly bounded diameter.

If $(X,d)$ is a metric space, then a \emph{uniform re-parameterized quasi-geodesic} is a function $f:\mathbb{R} \to X$ such that there are $\ldots < n_{-1}<n_0<n_1<\ldots \in \mathbb{R}$ such that $g:\mathbb{Z} \to X:i \mapsto f(n_i)$ is a uniform quasi-geodesic and $f([n_i,n_{i+1}])$ is uniformly bounded .  We have the following:

\begin{prop}\label{liberal quasi-geodesic}
 If $T_t$ is a liberal folding path, then $\hat{\pi}(T_t)$ is a reparameterized uniform quasi-geodesic.  
\end{prop}

\begin{proof}
 Theorem A.12 of \cite{BF14} gives that $T_t$ is a reparameterized uniform quasi-geodesic in the simplicial completion $(\mathcal{S},d_{\mathcal{S}})$ of $CV_N$, equipped with the simplicial metric.  The space $\mathcal{S}$ is called the \emph{free splitting complex}.  Theorem 1.1 of \cite{KR12} then gives that $\pi(T_t)$ is a reparameterized uniform quasi-geodesic in $\mathcal{FF}$.  
\end{proof}

The follow result of Bestvina and Feighn generalizes \cite{AK11}:

\begin{prop}\label{strongly contracting}\cite[Corollary 7.3]{BF11}
 Let $U, V \in CV_N$, and suppose that $T_t$ is the image in $CV_N$ of a greedy folding path and has been parameterized by arc length.  Further, suppose that $d(U,T_t)>M$ for all $t$ and that $d(U,V) \leq M$.  If $\pi(T_t)$ is a uniform quasi-geodesic, then $\text{right}(U)-\text{left}(U)$ and
 \[
 \sup_{s\in [\text{left}(U),\text{right}(U)], t\in [\text{left}(V),\text{right}(V)]}d(T_s,T_t)
 \]
 are uniformly bounded.
\end{prop}

Notice that it certainly is implied that in this case there is $\epsilon>0$ such that $T_t \in CV_N^\epsilon$.   Geodesics satisfying the conclusion of Proposition \ref{strongly contracting} are called \emph{strongly contracting}.  The following is a consequence of the work of Bestvina-Feighn; we are including a proof for the convenience of the reader.

\begin{cor}\label{uniformly stable}\cite{BF11}
 Suppose that $T_t$ is a strongly contracting folding path in $CV_N$.  If $T,U \in CV_N$ are such that $d(T,\text{Image}(T_t))$ and $d(U,\text{Image}(T_t))$ are bounded, then any geodesic from $T$ to $U$ lies in a bounded $d_s$-neighborhood of $\text{Image}(T_t)$.
\end{cor}

\begin{proof}
 In light of Proposition \ref{strongly contracting} and Lemma \ref{thick distance}, supposing the conclusion is false would produce a contradiction to the triangle inequality.  
\end{proof}

\subsection{The boundary of the factor complex}

Additional details about the subject matter of this section can be found in \cite{BR12, R12}.

Use $\bar{CV}_N$ to denote the space of homothety classes of elements of $\overline{cv}_N$; this is the usual compactification of Outer space \cite{CV86}.  Put $\partial cv_N=\bar{cv}_N \ssm cv_N$ and $\partial CV_N=\overline{CV}_N \ssm CV_N$.  

Associated to $T \in \partial cv_N$ is an $\FN$-invariant, closed subspace $L(T) \subseteq \partial^2 \FN$ called the \emph{lamination} of $T$ defined as follows:
\[
L(T)=\cap_{\delta>0} \overline{\{(g^{-\infty},g^\infty)|l_T(g)< \delta\}}
\]

Where $g^{\pm \infty}\in \partial \FN$ is $g^{\pm \infty}=\lim_{n \to \pm \infty} g^n1$.  Invariance of $L(T)$ follows from the fact that $l_T(\cdot)$ is constant on conjugacy classes.  See \cite{CHL08a, CHL08b}.

Given $T \in \partial cv_N$ and a factor $F'$, say that $F'$ \emph{reduces} $T$ if there is an $F'$-invariant subtree $Y \subseteq T$ with $vol(Y/F')=0$.  This means that the action of $F'$ on $Y$ has a dense orbit; notice that $Y$ could be a point.  Define $\mathcal{R}(T)=\{F'|F' \text{ is a factor reducing } T\}$.

We use the notation $\mathcal{AT}=\{T \in \partial cv_N|\mathcal{R}(T)=\emptyset\}$, and the elements of $\mathcal{AT}$ are called \emph{arational trees}.  It is shown in \cite{R12} that $T \in \mathcal{AT}$ if and only if either $T$ is free and \emph{indecomposable} in the sense of Guirardel \cite{Gui08} or else $T$ is dual to an arational measured foliation on a once-punctured surface.  

Put a relation $\sim$ on $\mathcal{AT}$ by declaring $T \sim U$ if and only if $L(T)=L(U)$.  The boundary $\partial \mathcal{FF}$ of $\mathcal{FF}$ was considered in \cite{BR12}; see also \cite{H12}.  We have the following:

\begin{prop}\label{delF}\cite[Main Results]{BR12}
 There is a (closed) quotient map $\partial \pi:\mathcal{AT} \to \partial \mathcal{FF}$ such that $\partial \pi(T)=\partial \pi(U)$ if and only if $T \sim U$.  If $T_n \in cv_N$ converge to $T \in \partial cv_N$, then:
 \begin{enumerate}
 \item [(i)] If $T \notin \mathcal{AT}$, then $\hat{\pi}(T_n)$ does not converge in $\mathcal{FF}$, and
 \item [(ii)] If $T \in \mathcal{AT}$, then $\hat{\pi}(T_n)$ converges to $\partial \pi(T)$.  
 \end{enumerate}
 Furthermore, if $T_n, U_n \in cv_N$ converge to $T,U \in \mathcal{AT}$, if $T \nsim U$, and if $V_t^n$ is a greedy folding path from $\lambda_nT_n$ to $U_n$, then $V_t^n$ converges uniformly on compact subsets to a greedy folding path $V_t$ with $\lim_{t \to -\infty} V_t/\mu_t \sim T$ and $\lim_{t \to \infty} V_t \sim U$, for some constants $\lambda_n$, $\mu_t$.
\end{prop}

\begin{rem}
 If $(V_n)_n$ is a reduced unfolding path with $\lim_{n \to -\infty} V_t=T$, then Lemma \ref{length and frequency decay} gives that $\Lambda((V_n)_n) \subseteq L(T)$.  
\end{rem}

Recall that $Curr(\FN)$ denotes the space of currents on $\FN$ and that $Curr(\FN)$ is given the weak-* topology.     When $g \in \FN$ is non-trivial and when no element of the form $h^k$, $k>1$, belongs to the conjugacy class of $g$, the notation $\eta_g$ is used to mean the sum of all Dirac masses on the distinct translates of $(g^{-\infty},g^\infty)$ in $\partial^2 \FN$.  The following useful results were established by Kapovich and Lustig:

\begin{prop}\cite[Main Results]{KL09a, KL10d}\label{KL}
 There is a unique $Out(\FN)$-invariant continuous function $\langle \cdot, \cdot \rangle: \overline{cv}_N \times Curr(\FN) \to \mathbb{R}_{\geq 0}$ that satisfies:
 \begin{enumerate}
 \item [(i)] $\langle T, \eta_g\rangle=l_t(g)$,
 \item [(ii)] $\langle T, \mu \rangle =0$ if and only if $\text{Supp}(\mu)\subseteq L(T)$
 \end{enumerate}
\end{prop}

Let $g \in \FN$ be a primitive element, \emph{i.e.} $\langle g \rangle$ represents a factor.  Define $\mathbb{R}M_N=\overline{Out(\FN)[\eta_g]}$, where $[\eta_g]$ is the homothety class of $\eta_g$ in the projective space $\mathbb{P}Curr(\FN)=Curr(\FN)/\mathbb{R}_{\geq 0}$.  The subspace $\mathbb{P}M_N$ is the minset of the action of $Out(\FN)$ on $\mathbb{P}Curr(F_N)$ \cite{MarPhD, KL07}.  Use $M_N$ to denote the pre-image of $\mathbb{P}M_N$ in $Curr(\FN)$.

If $X$ is a topological space, then $X'$ denotes the set of non-isolated points in $X$, $X''=(X')'$, and so on.  We have the following, which is the contents of Proposition 4.2, Corollary 4.3, and Theorem 4.4 of \cite{BR12}:

\begin{prop}\label{uniqueduality}\cite{BR12}
 Let $T \in \mathcal{AT}$ and $U \in \partial cv_N$
 \begin{enumerate}
 \item [(i)] $L'''(T)$ is perfect,
 \item [(ii)] If $L'''(T) \subseteq L(U)$, then $L(T)=L(U)$, and
 \item [(iii)] If $\mu \in M_N$ satisfies $\langle T,\mu \rangle=0$, then $Supp(\mu)=L'''(T)$.
 \end{enumerate}
 In particular, if $\langle T, \mu \rangle = 0 = \langle U, \mu \rangle$, then for any $\eta \in M_N$, one has that $\langle T, \mu \rangle=0$ if and only if $\langle U, \mu \rangle=0$, so $U \in \mathcal{AT}$.
\end{prop}

If $L_n$ and $L$ are algebraic laminations, say that the sequence $L_n$ \emph{super-converges} to $L$ if for any $l \in L$, there is a sequence $l_n \in L_n$ such that $l_n$ converges to $l$. It is immediate from the definition of $L(\cdot)$ and the definition of the topology on $\overline{cv}_N$ that for $T_n,T \in \partial cv_N$, if $T_n$ converge to $T$, then $L(T_n)$ super-converges to $L(T)$.

For $T \in \mathcal{AT}$ define $\Lambda(T)=L'''(T)$, and use the notation $\mathcal{AL}=\{\Lambda(T)|T \in \mathcal{AT}\}$.  Call the topology induced by the surjection $\mathcal{AT} \to \mathcal{AL}$ the \emph{super-convergence topology}.  A rephrasing of the first statement of Proposition \ref{delF} is:

\begin{cor}\label{delF2}
 $\partial \mathcal{FF} \cong \mathcal{AL}$
\end{cor}

Define $\mathcal{UET}=\{T \in \mathcal{AT}|L(T)=L(T') \text{ imples } T'=\alpha T, \alpha \in \mathbb{R}\}$ and $\mathcal{UEL}=\{\Lambda \in \mathcal{AL}|\text{Supp}(\eta)=\Lambda=\text{Supp}(\eta'), \eta \in M_N, \text{ implies } \eta=\alpha\eta, \alpha \in \mathbb{R}\}$.  Finally set $\mathcal{UE}=\mathcal{UET} \cap \{T |\Lambda(T) \in \mathcal{UEL}\}$.

\subsection{Limit Sets of Certain Subgroups of $Out(F_N)$}

If $H \leq Out(F_N)$, we use $\text{Limit}(H) \subseteq \partial cv_N$ to mean representatives of accumulation points of $x_0H$ in $\overline{CV}_N$, where $x_0$ is any base point.  

\begin{thm}\label{Limit Set}
 Let $H \leq Out(F_N)$.  If the orbit map $H \to \mathcal{FF}:h  \mapsto h\pi(x_0)$ is a quasi-isometric embedding, then $\text{Limit}(H)\subseteq \mathcal{UE}$.
\end{thm}

\begin{proof}
 First notice that the conclusion of Corollary \ref{uniformly stable} also holds for uniform quasi-geodesics from $T, U$.  It follows that geodesics between points in $x_0H$ project to uniform quasi-geodesics in $\mathcal{FF}$.  Now apply Proposition \ref{delF} to get that bi-infinite geodesics arising as limits of geodesics between points in $x_0H$ also project to uniform quasi-geodesics in $\mathcal{FF}$.  
 
 Applying Proposition \ref{delF} again, we get that for any $U \in \text{Limit}(H)$, there are geodesics $T_t, \check{T}_t$ with $\lim_{t \to \infty} T_t \sim U$ and $\lim_{t \to -\infty}\check{T}_t \sim U$.  Now apply Theorem \ref{slow progress in FF for foldings} to get that $\lim_{t\to \infty} T_t \in \mathcal{UET}$, which implies that $\lim_{t \to \infty} T_t=U$ and $\lim_{t \to -\infty} \check{T}_t=U$.  Now apply Theorem \ref{slow progress in FF for unfoldings} to get that $\Lambda(\lim_{t \to -\infty} \check{T}_t) \in \mathcal{UEL}$.  
\end{proof}

\subsection{Random walks on $Out(F_N)$}

The goal of this section is to combine our results about folding/unfolding paths and our background material with recent work of Tiozzo \cite{Tio14} and Maher-Tiozzo \cite{MT14} to study random walks on $Out(\FN)$.  Our goal is the following application:

\begin{thm}\label{Poisson boundary}
 Let $\mu$ be a non-elementary distribution on $Out(\FN)$ with finite first moment with respect to $d_s$.  Almost every sample path for the corresponding random walk converges to a point $T \in \partial cv_N$, and the corresponding harmonic measure $\nu$ is concentrated on the subspace of uniquely ergodic and dually uniquely ergodic arational trees and is the Poisson boundary.  Furthermore, for almost every sample path $(w_n)_n$, there is a Lipschitz geodesic $T_t$ with image $\gamma$, such that 
 \[
 \lim d_s(w_n0, \gamma)/n=0
 \]
\end{thm}

First, we define framework and will follow the discussion in \cite{MT14}.  We use $\Gamma$ to denote a finitely generated discrete group.

 A \emph{distribution} on $\Gamma$ is a positive vector $\mu \in L^1(\Gamma)$ of norm one.  The associated \emph{step space} is $\Omega=(\Gamma^\mathbb{Z},\mu^\mathbb{Z})$, and we use $\mathbb{P}$ to denote the product measure $\mathbb{P}=\mu^\mathbb{Z}$.  The \emph{shift map} on $\Omega$ is $S:\Omega \to \Omega:(g_n) \mapsto (g_{n+1})$; $S$ is ergodic.   
 
 Given a $\Gamma$-space $X$ and base point $x_0 \in X$, one obtains a discrete random variable with values in $X$ via $\Gamma^\mathbb{Z} \ni (g_n) \mapsto w_nx_0$, which is called a \emph{random walk} on $X$; we consider the push forward of $\mathbb{P}$ on the orbit $\Gamma x_0$.  The \emph{location} $(w_n)$ of the random walk at time $n$ is:
 \[
 w_n = \left\{
        \begin{array}{ll}
            g_1g_2\ldots g_n & \quad n>0 \\
            1 & \quad n=0 \\
            g_{0}^{-1}g_{-1}^{-1}\ldots g_{n}^{-1} & \quad n<0\\
        \end{array}
    \right.
 \]
 
 For negative time, the behavior of paths is guided by the \emph{reflected measure} $\check{\mu}(g)=\mu(g^{-1})$.  The sequences $(w_n)_{n>0}$ and $(w_n)_{n<0}$ are called \emph{unilateral paths} and are governed by the measures $\mu$ and $\check{\mu}$ respectively; we also say that $(w_n)_{n>0}$, \emph{resp.} $(w_n)_{n<0}$, is a \emph{sample path} for $\mu$, \emph{resp.} $\check{\mu}$.  When we use the term sample path without reference to a measure $\mu$ or $\check{\mu}$, we shall mean a unilateral path $(w_n)_{n>0}$.  The paths $(w_n)$ are called \emph{bilateral paths}.  The notation $\textbf{w}$ may be used for a sample path.
 
 When $(X,\rho)$ is a metric space, one can consider the average step size, which is the first moment of $\mu$ with respect to $\rho$:
 \[
 \sum_{g \in \Gamma}\mu(g)\rho(0,g0)
 \]
 
 If $Y$ is a topological space and if $\Gamma$ has a Borel action on $Y$, then $\Gamma$ acts on the space of Borel probability measures $\mathcal{M}(Y)$ via $g\nu(A)=\nu(g^{-1}(A))$.  The space $\mathcal{M}(Y)$ gets the weak-* topology, and it is a standard fact that compactness of $Y$ implies compactness of $\mathcal{M}(Y)$.  The distribution $\mu$ gives rise to the convolution operator on $\mathcal{M}(Y)$: for $\nu \in \mathcal{M}(Y)$
 \[
 \mu \ast \nu(A)=\sum_{g \in \Gamma}\mu(g)g\nu(A)
 \]
 A measure $\nu$ on $Y$ is called $\mu$-\emph{stationary} if it is a fixed point of the convolution operator.  Use $\mu^{(n)} \ast \nu$ to denote the $n$-fold application of the convolution to $\nu$.  A measure $\nu \in \mathcal{M}(Y)$ is called a $\mu$-\emph{boundary} if $\nu$ is stationary and for almost every sample path $(w_n)$, one has that $\lim_n w_n\nu$ is a Dirac mass.  The maximal $\mu$-boundary is called the \emph{Poisson boundary}.

 \begin{lem}\label{AT Borel}
  $\mathcal{AT}$ is a Borel set.
 \end{lem}
 
 \begin{proof}
  Let $\{F^1,F^2,\ldots,F^k,\ldots\}$ be the list of all factors; use $\mathcal{R}_k=\{T \in \partial cv_N|\exists F' \in \mathcal{R}(T), F' \leq F^k\}$.  By definition, $\mathcal{AT}=\cap_k (\partial cv_N \ssm \mathcal{R}_k)$.  We will finish by showing that $\mathcal{R}_k$ is closed.  Note that $\partial cv_N$ is separable and metrizable.  
  
  Suppose that $T_j \in \mathcal{R}_k$ converge to $T \in \partial cv_N$.  If $F^k$ fixes a point on a subsequence of $T_j$, then by definition of the topology on $\overline{cv}_N$, then $F^k$ fixes a point in $T$ as well.  Hence, we assume that $(T_j)_{F^k}$ is not a point.  By the definition of $\mathcal{R}_k$, we have that $(T_j)_{F^k} \in \partial cv(F^k)$, where $\overline{cv}(F^k)$ is the space of very small $F^k$-trees.  Further, notice that the definition of the topology of $\overline{cv}_N$ ensures that the restriction map $(T_j) \mapsto (T_j)_{F^k}$ is continuous.  The fact that $\mathcal{R}_k$ is closed then follows from the fact that $\partial cv(F^k)$ is closed.
 \end{proof}
 
 For $T \in \mathcal{AT}$, use $[T]$ to denote the $\sim$-class of $T$.  Define $B =\mathcal{AT} \times \mathcal{AT} \ssm \text{Graph}(\sim)$.  For $(T,U) \in B$, define $\mathbb G(T,U)=G([T],[U])$ to be the set of all geodesics $G_t$ in $CV_N$ with $\lim_{t \to -\infty} G_t \in [T]$ and $\lim_{t \to \infty} G_t \in [U]$.  By Proposition \ref{delF}, $\mathbb G(T,U) \neq \emptyset$.  If $G_t \in \mathbb{G}(T,U)$, use $G$ for the image of $G_t$.  Fix a base point $x_0 \in cv_N$.  Define a function $\Phi:B \to \mathbb{R}$ by
 \[
 \Phi(T,U)=\sup_{G_t\in \mathbb G(T,U)} d_s(x_0,G)
 \]
 
 \begin{lem}
  $\Phi(T,U)$ is defined for every $(T,U) \in B$.  
 \end{lem}
 
 \begin{proof}
 Let $G_t^n \in \mathbb{G}(T,U)$ be a sequence.  According to \cite[Corollary 6.8]{BR12}, we have that the sequence of folding path $G_t^n|_{[-n,n]}$ must accumulate on a point $y \in cv_N$.  It follows that $\Phi(T,U)$ is defined.  
 \end{proof}
 
 \begin{lem}\label{phi is measurable}
  $\Phi$ is a Borel function.
 \end{lem}
 
 \begin{proof}
  Suppose $(T_i,U_i) \in B$ converge to $(T,U) \in B$.  Choose $s_i \to -\infty$ and $t_i \to \infty$ and apply Proposition \ref{delF} to greedy folding paths $G_t^i$, $s_i \leq t \leq t_i$.  After passing to a subsequence we get uniform convergence on compact sets to $G_t \in \mathbb{G}(T,U)$.  By definition of $\Phi$, we then have that $\lim_i \Phi(U_i,T_i) \leq \Phi(U,T)$, and so $\Phi$ is Borel.  
 \end{proof}
 
 If $\Gamma$ acts on a hyperbolic space $X$, then we say that a distribution $\mu$ is \emph{non-elementary} (relative to the $\Gamma$-action on $X$) if the semigroup generated by the support of $\mu$ does not fix a finite subset of $X \cup \partial X$.  Notice that $\mu$ is non-elementary if and only if $\check{mu}$ is non-elementary.
 
 \begin{rem}
  By the subgroup classification theorem of Handel-Mosher \cite{HM09} and Horbez \cite{Hor14a}, one has that a distribution $\mu$ on $Out(\FN)$ is non-elementary with respect to the action on $\mathcal{FF}$ if and only if 
 \end{rem}
 
 Recently, Maher and Tiozzo have shown the following:
 
 \begin{prop}\cite[Theorems 1.1, 1.2]{MT14}\label{Tiozzo Maher}
  Let $\Gamma$ be a group that acts on a separable hyperbolic space $(X,d_X)$ with base point $0$.  Assume that $\mu$ is a non-elementary distribution on $\Gamma$ with finite first moment.  
  \begin{enumerate}
  \item [(i)] For almost every sample path $(w_n)$, one has that $w_n0$ converges in $X$, and the corresponding hitting measure on $\partial X$ is non-atomic and is the unique $\mu$-stationary measure.
  \item [(ii)] There is $L_0>0$ such that for almost every sample path $(w_n)$, one has 
  \[
  \lim_n d_X(0,w_n0)/n=L_0
  \]
  \end{enumerate}
 \end{prop}
 
 We will also need the following simple lemma, which is immediate from the triangle inequality and Kingmann's Sub-additive Ergodic Theorem:
 
 \begin{lem}\label{escape rate}
  Let $\mu$ be a distribution on $\Gamma$, and let $(X,d_X)$ be a $\Gamma$-metric space with base point $0$.  If $\mu$ has finite first moment, then there is $L_1$ such that for almost every sample path $(w_n)$, one has
  \[
  \lim_n d_X(0,w_nx)/n=L_1
  \]
 \end{lem}
 
 The following is established by Tiozzo:
  
\begin{lem}\label{Tiozzo lemma}\cite[Lemma 7]{Tio14}
 Let $\Omega$ be a measure space with a probability measure $\lambda$, and let $T : \Omega \to \Omega$ be a measure-preserving, ergodic transformation. Let $f:\Omega \to \mathbb{R}$ be a non-negative, measurable function, and define the function $g : \Omega \to \mathbb{R}$ as
 \[
g(\omega) := f (T \omega) - f (\omega) \quad \forall \omega \in \Omega 
 \]
 If $g \in L^1(\Omega, \lambda)$, then, for $\lambda$-almost every $\omega \in \Omega$, one has
 \[
\lim f(T^n \omega)/n=0
 \]  
\end{lem} 
 
 \begin{thm}\label{mu boundary}
  Let $\mu$ be a non-elementary distribution of $Out(\FN)$ with finite first moment with respect to $d_s$.  Then
  \begin{enumerate}
  \item [(i)] There exists a unique $\mu$-stationary probability measure $\nu$ on the space $\partial CV_N$, which is purely non-atomic and concentrated on $\mathcal{UE}$; the measure space $(\partial CV_N, \nu)$ is a $\mu$-boundary, and
  \item [(ii)] For almost every sample path $(w_n)$, $w_n0$ converges in $\overline{CV}_N$ to $w_\infty \in \mathcal{UE}$, and $\nu$ is the corresponding hitting measure.
  \end{enumerate}
 \end{thm}
 
 \begin{proof}
  Notice that finite first moment with respect to $d_s$ certainly gives finite first moment with respect to $d_{\mathcal{FF}}$ by Proposition \ref{projection properties}.  According to Proposition \ref{Tiozzo Maher}, there is a unique $\mu$-stationary, \emph{resp.} $\check{\mu}$-stationary, purely non-atomic measure $\cover \nu_+$, \emph{resp.} $\cover \nu_{-}$, on $\partial \mathcal{FF}$.  By Proposition \ref{delF}, we have that $\partial \mathcal{FF} \cong \mathcal{AT}/\sim$; in particular, $\mathcal{AT} \ni T \mapsto [T] \in \partial \mathcal{FF}$ is continuous.  Further, by Lemma \ref{AT Borel}, we have that $\mathcal{AT}$ is measurable.  
  
  Let $\nu_+$, \emph{resp.} $\nu_{-}$, be a $\mu$-stationary, \emph{resp.} $\check{\mu}$-stationary, measure on $\partial CV_N$; the existence of $\nu_+$, $\nu_{-}$ is guaranteed by compactness of $\partial CV_N$; see, for example, \cite[Lemma 2.2.1]{KM96}.  By the previous paragraph, $\nu_+$, $\nu_{-}$ must be concentrated on $\mathcal{AT}$, and $(\partial \mathcal{FF}, \cover{\nu}_+)$, \emph{resp.} $(\partial \mathcal{FF}, \cover \nu_{-})$ is a quotient space of $(\partial CV_N, \nu_+)$, \emph{resp.} $(\partial CV_N, \nu_{-})$.  It follows that $\nu_+,\nu_{-}$ are non-atomic.  We will show that these quotients are in fact an isomorphisms by showing that $\nu_+$, $\nu_{-}$ are concentrated on $\mathcal{UE}$.  
  
  By \cite[Lemma 2.2.3]{KM96} for almost every sample path $(w_n)$ (corresponding to $\textbf{g}$), the weak-* limits $\lim_{n \to \pm \infty} w_n\nu=\lambda(w_{\pm \infty})$ exist, and by Proposition \ref{Tiozzo Maher}, $\lambda(w_{\pm \infty})$ are each  concentrated on a single $\sim$-class in $\mathcal{AT}$, which we denote by $\textbf{bnd}_{\pm}(\textbf{g})$.

\begin{claim} 
 For almost every sample path $(w_n)$ there is a Lipschitz geodesic $T_t$ with endpoints $T,U \in \mathcal{AT}$ such that 
 \[
 \lim d_s(w_n0,T_{L_1n})/n=0
 \]
\end{claim}

The proof follows \emph{verbatim} Tiozzo's proof of \cite[Theorem 6]{Tio14}; in particular, Tiozzo never uses that $\partial X$ in the statement of Theorem 6 is in any way related to $X$.  Tiozzo's proof is included for the convenience of the reader.

\begin{proof}\cite{Tio14}
 We apply Lemma \ref{Tiozzo lemma} to the space of increments with the shift operator and with $f$ defined by
 \[
 f(\textbf{g}):=\Phi(\textbf{bnd}_{-}(\textbf g),\textbf{bnd}_+(\textbf g))
 \]
 For $\textbf{g} \in Out(\FN)^\mathbb{Z}$, we have: 
 \[
 f(S^k\textbf{g})=\sup_{G \in \mathbb G(\b{bnd}_-(S^k\b g),\b{bnd}_+(S^k\b g)} d(0,G)
 \] 
It is easy to see that $\b{bnd}_{\pm}(S^k \b g)=w_k^{-1}\b{bnd}_{\pm}(\b g)$, so 
\[
f(S^k \b g)=\sup_{G \in \mathbb G(\b{bnd}_-(\b g),\b{bnd}_+(\b g)} d(w_k0,G)
\]
It is evident from the definition that $\mathbb G(\cdot, \cdot)$ is equivariant, hence so is $\Phi$.  Hence, we have that $g(\b g) \leq d(0,g_1^{-1}0)$ and the first moment assumption gives that $g \in L^1(Out(\FN)^\mathbb{Z},\mathbb P)$.  Hence, Lemma \ref{Tiozzo lemma} gives 
\[
\lim f(S^n \b g)/n=0
\]
almost surely.  This certainly means that for any $G \in \mathbb G(\b{bnd}_-(\b g),\b{bnd}_+(\b g)$ we have that 
\[
\lim d(w_k0,G)/n=0
\]
Suppose that $G=T_t$, then we have times $t_k$ such that 
\[
\lim d(w_k0,T_{t_k})/n=0
\]
By Lemma \ref{escape rate}, we must have have that $t_k/t=\pm L_1$.  
\end{proof} 

Recall that from Proposition \ref{Tiozzo Maher} we have $L_0>0$ such that almost surely $\lim d_{\mathcal{FF}}(0,w_n0)/n=L_0$.  Combining this with Lemma \ref{escape rate}, we have $L_1>0$ such that $\lim d_s(0,w_n0)/n=L_1$.  By the triangle inequality and Lemma \ref{projection properties}, we have
\[
d_s(0,T_{L_1n}) \geq d_s(0,w_n0)-d_s(T_{L_1n},w_n0)
\]
So
\[
\lim d_{\mathcal{FF}}(\pi(0),\pi(T_{L_1n}))/n = L_0
\]
In particular, the positive ray of $T_t$ projects quasi-isometrically to $\mathcal{FF}$.  If we were to assume that $U=\b{bnd}_+(\b g) \notin \mathcal{UET}$, then we would obtain a contradiction from Theorem \ref{slow progress in FF for foldings}.  

By applying exactly the same string of arguments to $\b{bbd}_-(\b g)$ except using Theorem \ref{slow progress in FF for unfoldings} in place of Theorem \ref{slow progress in FF for foldings}, we get that $\Lambda(T) \in \mathcal{UEL}$.  Hence, we have established that that $\nu_{\pm}$ are concentrated on $\mathcal{UE}$.  Hence, the map $\partial \pi$ provides an isomorphism of $(\partial CV_N, \nu_{\pm})$ and $(\mathcal{AT}/\sim, \cover{\nu}_{\pm})$.

 \begin{claim}\label{convergence in cv}
 For almost every sample path, one has that $w_n0$ converges in $\overline{CV}_N$.  
\end{claim}

\begin{proof}
 We have that $w_n0$ tracks sublinearly along a geodesic $T_t$ that converges to $T \in \mathcal{UET}$.  Observe that $T_t/e^t$ converges in $\overline{cv}_N$.  For a conjugacy class $g$, use $|g|$ to denote its length in $T_0$.  Let $g_n$ be an embedded loop in $T_n$, then $\eta_{g_n}/|g_n|$ converges in $Curr(\FN)$ to $\eta \in M_N$ with $\langle T, \eta \rangle =0$; by Proposition \ref{uniqueduality}, we have that $Supp(\eta)=\Lambda(T)$.  
 
If $U \in \partial cv_N$ satisfies $\langle U, \eta \rangle =0$, then by Proposition \ref{KL}, we have that $\Lambda(T) \subseteq L(U)$.  Again applying Proposition \ref{uniqueduality}, we have that $L(U)=L(T)$, and since $T \in \mathcal{UET}$, we get that $U=T$.  Hence, we will finish by showing that $\lim_{n \to \infty} \langle w_n0, \eta \rangle=0$.
 
 Set $U_n=w_n0$.  We have $\langle U_n, \eta_{g_n} \rangle \leq \langle T_n, \eta_{g_n} \rangle e^{d(T_n,U_n)}$.  Further, $d(T_0,T_n) \leq d(T_0,U_n)+d(U_n,T_n) \leq d(T_0, U_n) + C(\epsilon) d(T_n,U_n) +B(\epsilon)$ by Lemma \ref{thick distance}.  Hence,
 
 \begin{align*}
 \frac{\langle U_n, \eta_{g_n} \rangle}{|g_n|e^{d(T_0,U_n)}} & 
 \leq \frac{\langle U_n, \eta_{g_n} \rangle}{|g_n|e^{d(T_0,T_n)- C(\epsilon)d(T_n,U_n)-B(\epsilon)}} \\
  \, & \leq \frac{\langle T_n, \eta_{g_n} \rangle e^{d(T_n,U_n)}}{|g_n|e^{d(T_0,T_n)-C(\epsilon)d(T_n,U_n)-B(\epsilon)}} \\
  \, & = \frac{\langle T_n, \eta_{g_n} \rangle}{|g_n|e^{d(T_0,T_n)-o(n)}} \to 0
 \end{align*}
 
 \end{proof}
\end{proof}

We get:

\begin{thm}
 Let $\mu$ be a non-elementary distribution on $Out(\FN)$ with finite first moment with respect to the word metric.  The unique $\mu$-stationary measure $\nu$ on $\partial CV_N$ is the Poisson boundary.
\end{thm}

\begin{proof}
 Since Outer space has exponentially bounded growth, Theorem \ref{mu boundary} gives that the hypotheses of the Ray Criterion of Kaimanovich are satisfied \cite{Kai85}.
\end{proof}

\begin{rem}
We make two observations:
\begin{enumerate}
\item [(i)] We had the option as well to check the Strip criterion.  Indeed, from what we have shown and the contraction results of \cite{BF11}, one has that the $\mathbb{G}(T,U)$ is contained in a bounded neighborhood of any of its elements for generic $(T,U)$, and the construction of strips from \cite{KM96} can be carried out.  
\item [(ii)] The subspace of $\mathcal{AT}$ consisting of trees that are dual to foliation on a surface admits a countable partition $\{Z_i\}$ that satisfies $Z_i\Phi=Z_i$ or $Z_i\Phi \cap Z_i=\emptyset$ for any $\Phi \in Out(\FN)$; indeed, each tree is contained in a copy of $\mathcal{PML}(S)$ in $\partial CV_N$, where $S$ is a once-punctured surface with fundamental group $\FN$.  Then \cite[Lemma 2.2.2]{KM96} gives that either $\nu$ is supported on a finite union of these $\mathcal{PML}(S)$'s or else $\nu$ is concentrated on trees that are free and indecomposable.  
\end{enumerate}
\end{rem}

\bibliographystyle{amsplain}
\bibliography{REF}







{\sc \tiny \noindent
Hossein Namazi, Department of Mathematics, University of Texas\\
Alexandra Pettet, Department of Mathematics, University of British Columbia\\
Patrick Reynolds, Department of Mathematics, Miami University}
\vspace{0.1cm}

\end{document}